%% file: M2CP.tex
\documentclass[letterpaper,11pt,twoside,keywordsasfootnote,addressatend,noinfoline]{article}

\usepackage{fullpage}
\usepackage[english]{babel}
\usepackage{amssymb}
\usepackage{amsmath}
\usepackage{theorem}
\usepackage{epsfig}
\usepackage{subfigure}
\usepackage{imsart}

\theorembodyfont{\slshape}
\newtheorem{theor}{Theorem}
\newtheorem{propo}{Proposition}[section]
\newtheorem{lemma}[propo]{Lemma}
\newtheorem{defin}[propo]{Definition}

\newenvironment{proof}{\noindent{\scshape Proof.}}{\hspace{2mm} $\square$ \\}

\newcommand{\R}{\mathbb R}
\newcommand{\Z}{\mathbb Z}
\newcommand{\E}{\mathbb E}
\newcommand{\ep}{\epsilon}
\newcommand{\norm}[1]{|\!|#1|\!|}
\newcommand{\ind}{\hbox{{\small 1} \hspace*{-11pt} 1}}

\newcommand{\leb}{\frak L}

\DeclareMathOperator{\card}{card}
\DeclareMathOperator{\dist}{dist}

\begin{document}

\begin{frontmatter}

\title     {Two-scale multitype contact process: coexistence \\ in spatially explicit metapopulations}
\runtitle  {Two-scale multitype contact process}
\author    {Nicolas Lanchier}
\runauthor {Nicolas Lanchier}
\address   {School of Mathematical and Statistical Sciences, \\ Arizona State University, \\ Tempe, AZ 85287.}

\begin{abstract} \ \
 It is known that the limiting behavior of the contact process strongly depends upon the geometry of the graph on which
 particles evolve: while the contact process on the regular lattice exhibits only two phases, the process on homogeneous
 trees exhibits an intermediate phase of weak survival.
 Similarly, we prove that the geometry of the graph can drastically affect the limiting behavior of multitype versions
 of the contact process.
 Namely, while it is strongly believed (and partly proved) that the coexistence region of the multitype contact process
 on the regular lattice reduces to a subset of the phase diagram with Lebesgue measure zero, we prove that the coexistence
 region of the process on a graph including two levels of interaction has a positive Lebesgue measure.
 The relevance of this multiscale spatial stochastic process as a model of disease dynamics is also discussed.
\end{abstract}



\end{frontmatter}

%%%%%%%%%%%%%%%%%%%%%%%%%%%%%%%%%%%%%%%%%%%%%%%%%%%%%%%%%%%%%%%%%%%%%%%%%%%%%%%%%%%%%%%%%%%%%%%%%%%%%%%%%%%%%%%%%%%%%%%%%%%%%%%%%%%%%%%%%%

\section{Introduction}
\label{Sec: introduction}

\indent The multitype contact process introduced in \cite{neuhauser_1992} is a continuous-time Markov process whose state space maps
 the $d$-dimensional integer lattice into the set $\{0, 1, 2 \}$ where state 0 refers to empty sites and where state $i$, $i = 1, 2$,
 refers to sites occupied by a type $i$ particle.
 Denoting by $\eta_t$ the state of the system at time $t$ and by $\sim$ the binary relation indicating that two vertices are nearest
 neighbors, the evolution of the process at vertex $x$ is described by the transition rates
 $$ \begin{array}{rclrcl}
   c_{0 \,\to \,1} (x, \eta) & = & \displaystyle B_1 \ \sum_{x \sim y} \ \ind \{\eta (y) = 1 \} \hspace*{20pt} &
   c_{1 \,\to \,0} (x, \eta) & = & \delta_1 \vspace{4pt} \\
   c_{0 \,\to \,2} (x, \eta) & = & \displaystyle B_2 \ \sum_{x \sim y} \ \ind \{\eta (y) = 2 \} \hspace*{20pt} &
   c_{2 \,\to \,0} (x, \eta) & = & \delta_2 \end{array} $$
%  is described by the Markov generator $L$ defined on the set of cylinder functions by
%  $$ \begin{array}{rcl}
%   L f (\eta) & = &
%  \displaystyle \sum_{x \in \Z^d} \ \sum_{y \sim x} \ B_1 \,\ind \{\eta (x) = 0 \} \,\ind \{\eta (y) = 1 \} \ [f (\eta_{x, 1}) - f (\eta)] \vspace{4pt} \\ & + &
%  \displaystyle \sum_{x \in \Z^d} \ \sum_{y \sim x} \ B_2 \,\ind \{\eta (x) = 0 \} \,\ind \{\eta (y) = 2 \} \ [f (\eta_{x, 2}) - f (\eta)] \vspace{4pt} \\ & + &
%  \displaystyle \sum_{x \in \Z^d} \ [\delta_1 \ind \{\eta (x) = 1 \} + \delta_2 \ind \{\eta (x) = 2 \}] \ [f (\eta_{x, 0}) - f (\eta)] \end{array} $$
%  where $\eta_{x, i}$ is obtained from $\eta$ by assigning the value $i$ to vertex $x$.
 where $c_{i \,\to \,j} (x, \eta)$ is the rate at which the state of $x$ flips from $i$ to $j$.
 That is, type $i$ particles give birth through the edges of the lattice to particles of their own type at rate $B_i$ and
 die spontaneously at death rate $\delta_i$.
 If an offspring is sent to a site already occupied, the birth is suppressed.

\indent The multitype contact process has been introduced and completely studied when the death rates are equal
 by Neuhauser \cite{neuhauser_1992}.
 To fix the time scale, assume that $\delta_1 = \delta_2 = 1$, and to leave out trivialities, assume in addition that the
 birth rates are greater than the critical value of the basic contact process \cite{harris_1974}.
 Then, the type with the highest birth rate outcompetes the other type.
 In the neutral case when the birth rates are equal, the process clusters in dimension $d \leq 2$ while coexistence occurs
 in dimension $d \geq 3$.
 Here and after, coexistence means strong coexistence: there exists a stationary distribution with a positive density of
 type 1 and type 2.
 The long-term behavior of the process when the death rates are different remains an open problem but Neuhauser conjectured
 that her results extend to the general case provided one replaces the birth rate by the ratio of the birth rate to the
 death rate.
 In particular, it is believed that the coexistence region as a subset of the space of the parameters has Lebesgue
 measure zero.
 The existence of such a coexistence region is mathematically interesting but for obvious reasons it is irrelevant to
 explain why species coexist in nature.
 In order to identify mechanisms (more meaningful for biologists) that promote coexistence, recent studies have
 focused on modifications of the multitype contact process in which the coexistence region contains an open set of the parameters.
 It has been proved in different contexts that coexistence is promoted by
 spatial \cite{chan_durrett_2006, durrett_lanchier_2008, lanchier_neuhauser_2006}
 and temporal \cite{chan_durrett_lanchier_2009} heterogeneities.
 This article introduces the first example of a multitype contact process in which coexistence is produced by the geometry of
 the graph on which particles evolve.
 In some sense, our main result is analogous to the one of Pemantle \cite{pemantle_1992} which states that, in contrast with
 the contact process on the regular lattice, the contact process on homogeneous trees exhibits a phase of weak survival.
 In both cases, the geometry of the graph is responsible for creating new qualitative behaviors.

\indent To construct our process, we consider the $d$-dimensional lattice $\Z^d$ as a homogeneous graph with
 degree $2d$ where vertices are connected to each of their $2d$ nearest neighbors.
 Let $N$ be an odd positive integer.
 Then, we consider the following collection of hyperplanes:
 $$ H (i, j) \ = \ \{x \in \R^d : x^i = N / 2 + j N \} \qquad \hbox{for} \quad
    i = 1, 2, \ldots, d \quad \hbox{and} \quad j \in \Z $$
 where $x^i$ denotes the $i$th coordinate of $x$, and remove from the original graph all the edges that intersect one of
 these hyperplanes.
 This induces a partition of the lattice into $d$-dimensional cubes with length edge $N$ that we call \emph{patches}.
 See the left-hand side of Figure \ref{Fig: graph} for an illustration where edges drawn in dotted lines are the edges
 to be removed.
 Since the parameter $N$ is odd, each patch has a central vertex.
 To complete the construction, we draw a long edge between the centers of adjacent patches as indicated in the right-hand
 side of Figure \ref{Fig: graph}.
 The resulting graph can be seen as the superposition of two lattices that we call \emph{microscopic}
 and \emph{mesoscopic} lattices.
 Even though, for more convenience, we will prove all our results for this particular graph, our main coexistence result can
 be easily extended to more general graphs that we shall call \emph{two-scale graphs}.
 These graphs are described in details at the end of this section.

\indent To formulate the evolution rules, we write $x \sim y$ to indicate that vertices $x$ and $y$ are connected by a
 short edge, and $x \leftrightarrow y$ to indicate that both vertices are connected by a long edge.
 The evolution at vertex $x$ is then given by the following transition rates:
 $$ \begin{array}{rclrcl}
    c_{0 \,\to \,1} (x, \eta) & = & \displaystyle
       B_1 \ \sum_{x \leftrightarrow y} \ \ind \{\eta (y) = 1 \} \ + \ \beta_1 \ \sum_{x \sim y} \ \ind \{\eta (y) = 1 \} \hspace*{20pt} &
    c_{1 \,\to \,0} (x, \eta) & = & \delta_1 \vspace{4pt} \\
    c_{0 \,\to \,2} (x, \eta) & = & \displaystyle
       B_2 \ \sum_{x \leftrightarrow y} \ \ind \{\eta (y) = 2 \} \ + \ \beta_2 \ \sum_{x \sim y} \ \ind \{\eta (y) = 2 \} \hspace*{20pt} &
    c_{2 \,\to \,0} (x, \eta) & = & \delta_2. \end {array} $$
%  $$ \begin{array}{rcl}
%   L f (\eta) & = &
%  \displaystyle \sum_{x \in N \Z^d} \ \sum_{y \leftrightarrow x} \ B_1 \,\ind \{\eta (x) = 0 \} \,\ind \{\eta (y) = 1 \} \ [f (\eta_{x, 1}) - f (\eta)] \vspace{4pt} \\ & + &
%  \displaystyle \sum_{x \in N \Z^d} \ \sum_{y \leftrightarrow x} \ B_2 \,\ind \{\eta (x) = 0 \} \,\ind \{\eta (y) = 2 \} \ [f (\eta_{x, 2}) - f (\eta)] \vspace{4pt} \\ & + &
%  \displaystyle \sum_{x \in \Z^d} \ \sum_{y \leftrightarrow x} \ \beta_1 \,\ind \{\eta (x) = 0 \} \,\ind \{\eta (y) = 1 \} \ [f (\eta_{x, 1}) - f (\eta)] \vspace{4pt} \\ & + &
%  \displaystyle \sum_{x \in \Z^d} \ \sum_{y \leftrightarrow x} \ \beta_2 \,\ind \{\eta (x) = 0 \} \,\ind \{\eta (y) = 2 \} \ [f (\eta_{x, 2}) - f (\eta)] \vspace{4pt} \\ & + &
%  \displaystyle \sum_{x \in \Z^d} \ [\delta_1 \ind \{\eta (x) = 1 \} + \delta_2 \ind \{\eta (x) = 2 \}] \ [f (\eta_{x, 0}) - f (\eta)] \end{array} $$
 Note that in the expression of $c_{0 \,\to \,i} (x, \eta)$ the first sum is empty whenever vertex $x$ is not located at
 the center of a patch.
 We call this process the \emph{two-scale multitype contact process}.
 Also, we call the parameters $\beta_i$ and $B_i$ the \emph{microscopic} and \emph{mesoscopic} birth rates.
 The one-color version of this process has been introduced by Belhadji and Lanchier \cite{belhadji_lanchier_2008} as a spatially
 explicit model of metapopulation \cite{hanski_1999, levins_1969}.
 The objective there was to determine parameter values for which survival occurs.
 In contrast, the emphasis here is on whether both types coexist or one type outcompetes the other type.
 While our analysis of the single-species model in \cite{belhadji_lanchier_2008} did not reveal any major difference between the
 contact processes on regular lattices and the graph of Figure \ref{Fig: graph}, our analysis of the
 multispecies model shows that two-scale graphs, as opposed to regular lattices, promote the coexistence of the species.
 From now on, we assume that the parameters are chosen in such a way that each type survives in the absence of the other one
 and refer to \cite{belhadji_lanchier_2008} for explicit conditions of survival.
 Finally, note that when $N = 1$, i.e., patches reduce to a single vertex, the values of the microscopic birth rates are
 irrelevant and the two-scale multitype contact process reduces to the multitype contact process. %\eqref{eq: multitype}.
 Therefore, to avoid trivialities, we also assume that $N \neq 1$.

\begin{figure}[t]
\centering
\mbox{\subfigure{\epsfig{figure = 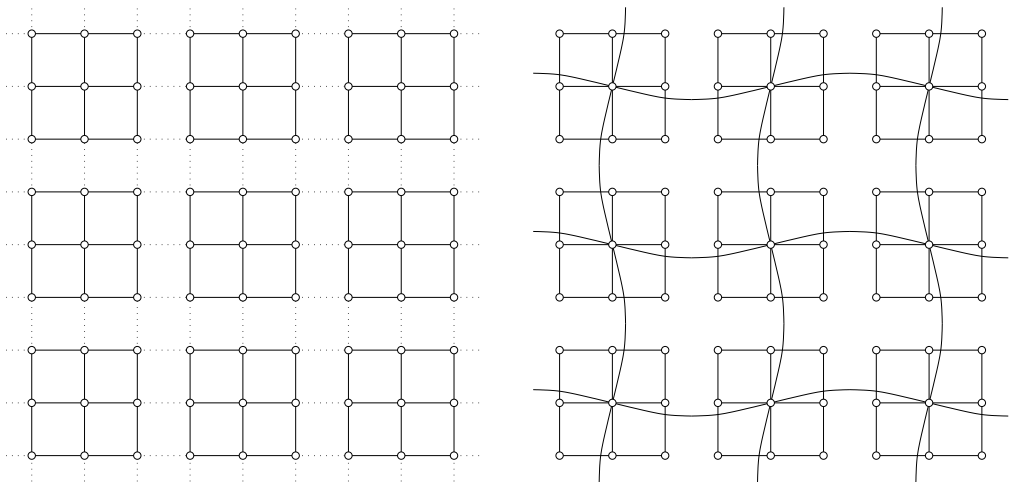, width = 380pt}}}
\caption{\upshape{Construction of the two-scale graph when $N = 3$ and $d = 2$.}}
\label{Fig: graph}
\end{figure}

\indent We call the best \emph{invader} the type with the highest $B_i$ to $\delta_i$ ratio, and the best \emph{competitor}
 the type with the highest $\beta_i$ to $\delta_i$ ratio.
 Our first theorem extends Neuhauser's result to the two-scale multitype contact process: assuming that the death
 rates are equal, when one type is both the best invader and the best competitor, it outcompetes the other type except in
 the neutral case when the process clusters in $d \leq 2$ and coexistence occurs in $d \geq 3$.
 See Figure \ref{Fig: M2CP} for pictures of numerical simulations in the two dimensional neutral case.

\begin{theor}[$\delta_1 = \delta_2 = 1$ and $N \neq 1$]
\label{duality}
 Assume that $B_1 \leq B_2$ and $\beta_1 \leq \beta_2$.
\begin{enumerate}
 \item In the neutral case $B_1 = B_2$ and $\beta_1 = \beta_2$ we have the following alternative.
\begin{enumerate}
\item In $d \leq 2$ clustering occurs, i.e., for any initial configuration,
 $$ \lim_{t \to \infty} P \,(\eta_t (x) = 1 \ \hbox{and} \ \eta_t (y) = 2) \ = \ 0 \quad \hbox{for all} \ x, y \in \Z^d. $$
\item In $d \geq 3$ coexistence occurs, i.e., there exists a stationary distribution under which the density of type 1 and the
 density of type 2 are both positive.
\end{enumerate}
 \item If $B_1 \leq B_2$ and $\beta_1 \leq \beta_2$ with at least one strict inequality then type 2 wins.
\end{enumerate}
\end{theor}

\begin{figure}[t]
\centering
\mbox{\subfigure{\epsfig{figure = 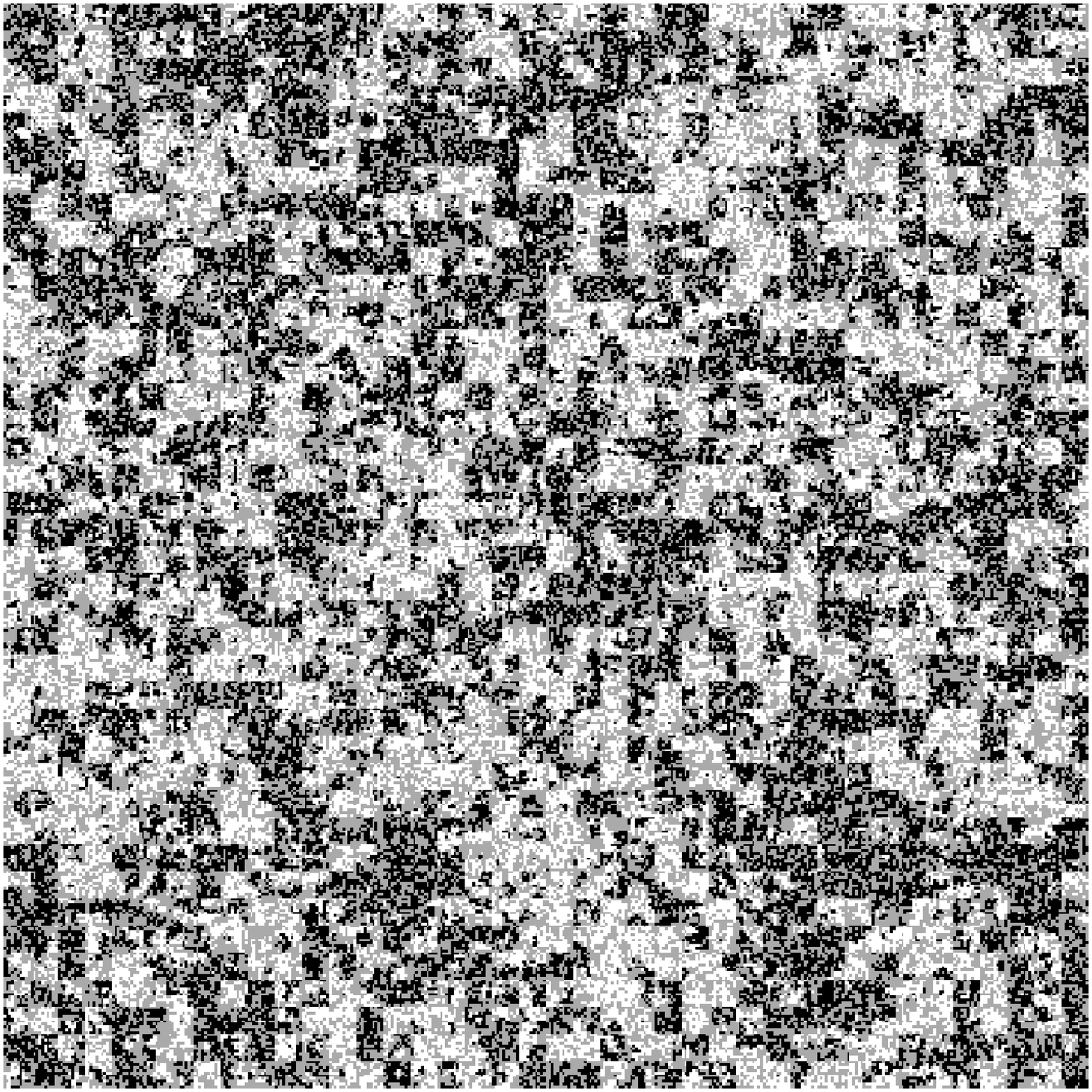, width = 200pt, height = 200pt}}} \hspace{10pt}
\mbox{\subfigure{\epsfig{figure = 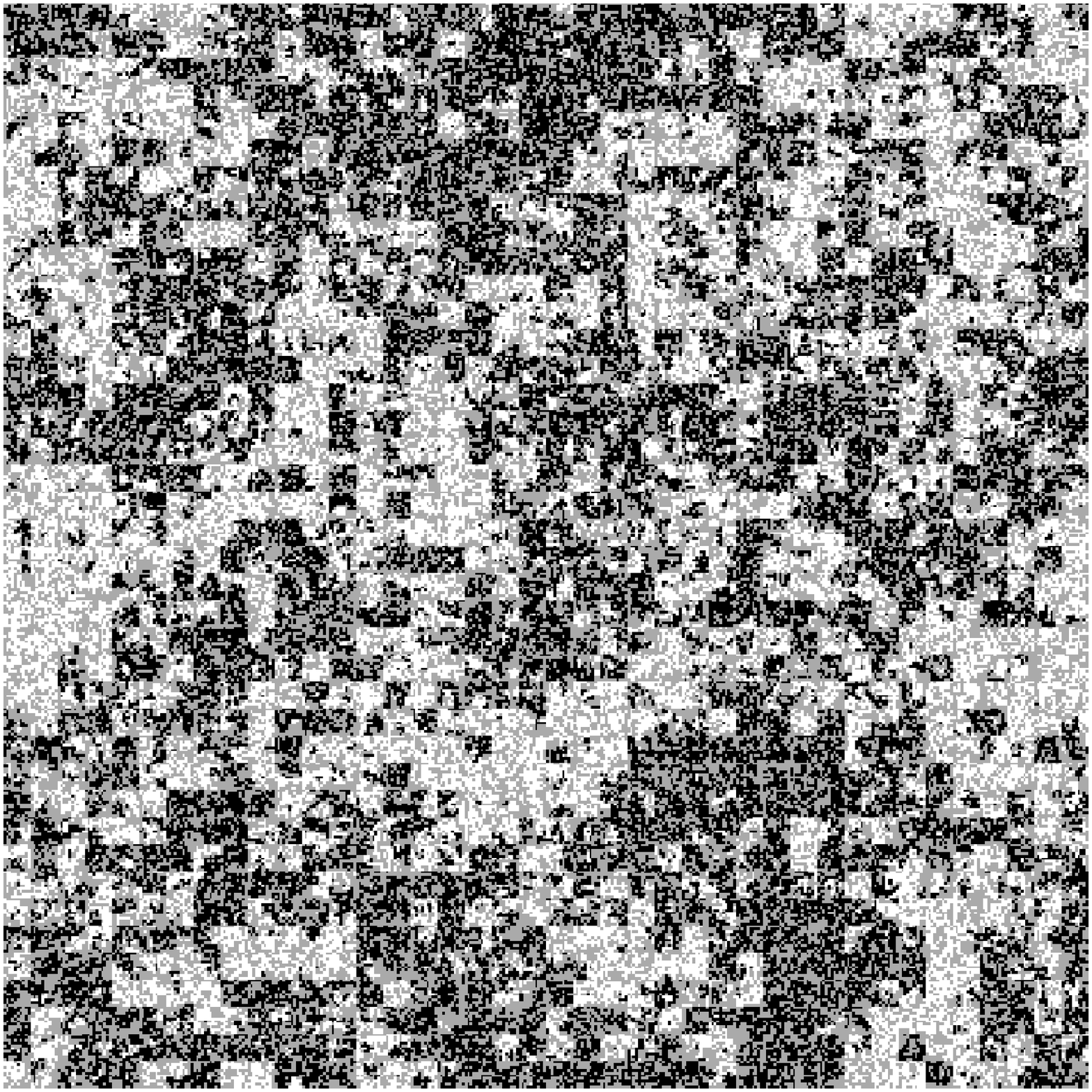, width = 200pt, height = 200pt}}}
\caption{\upshape Snapshots of the multitype two-scale contact process on the $400 \times 400$ lattice with periodic boundary
 conditions at times 100 and 500, respectively.
 White and black vertices refer to species 1 and 2, respectively, and grey vertices to empty sites.
 The parameters are $\beta_i = 3$ and $B_i = \delta_i = 1$ for $i = 1, 2$.}
\label{Fig: M2CP}
\end{figure}

\noindent To search for strategies promoting coexistence, we now assume that one type, say type 1, is the best invader,
 and the other type is the best competitor, in which case the limiting behavior of the process is more difficult
 to predict.
 Interestingly, while the results of Theorem \ref{duality} are not sensitive to the patch size, the long-term behavior of
 the process under these new assumptions strongly depends upon the parameter $N$.
 As mentioned above, when patches reduce to a single vertex, the values of the microscopic birth rates are irrelevant so
 that type 1 particles outcompete type 2 particles as predicted by Theorem 1 in \cite{neuhauser_1992}.
 In contrast, taking $N$ large leaves enough room for type 2 to outcompete locally type 1 within each patch, except maybe
 near the central vertices located on the mesoscopic lattice.
 The time a colony of type 2 particles persists within a single patch is long enough so that survival of type 2 is insured
 by casual migrations from one patch to another, which is referred to as the \emph{rescue effect} in metapopulation theory.
 In conclusion,
%  small patches promote survival of the best invader whereas
 large patches promote survival of the best competitor, as indicated in the following theorem.

\begin{theor}[$\delta_1 = \delta_2 = 1$]
\label{survival}
 Assume that $B_1 \geq B_2 > 0$ and $\beta_2 > \beta_1 > \beta_c$.
 Then, in any dimension, type 2 survives provided the spatial scale $N$ is sufficiently large.
\end{theor}

\noindent Theorem \ref{survival} is the key result to identify a set of parameters for which coexistence occurs.
 First of all, we fix the parameter values to make type 1 a good invader but a bad competitor living at a slow time scale.
 Type 1 then survives by jumping from patch to patch, or equivalently by invading the mesoscopic lattice.
 This holds regardless of the patch size.
 Coexistence of both types then follows from the proof of Theorem \ref{survival} by fixing the remaining parameters to make
 type 2 a good competitor living at a much faster time scale than type 1 and taking $N$ large (which does not affect
 survival of type 2).
 To prove rigorously survival of both types, the two-scale multitype contact process will be simultaneously coupled with
 two different oriented percolation processes, one following the evolution of type 1 at a certain time scale, the other
 one following the evolution of type 2 at a slower time scale.
 The use of a block construction implies the existence of an open set of the parameters in which coexistence occurs, so
 we can conclude that

\begin{theor}
\label{coexistence}
 In any dimension, the Lebesgue measure of the coexistence region is strictly positive provided the spatial scale $N$
 is sufficiently large.
\end{theor}

\noindent As previously mentioned, Theorem \ref{coexistence} together with Neuhauser's conjecture \cite{neuhauser_1992}
 indicates that, in contrast to the regular lattice, the graph of Figure \ref{Fig: graph} promotes coexistence for the
 multitype contact process.
 Thinking of the competitive exclusion principle in ecology (the number of coexisting species at equilibrium cannot
 exceed the number of resources), this suggests that the two-scale graph provides two spatial resources, namely the
 microscopic and the mesoscopic lattices, which allows two types to coexist.
 Note also that, while the coexistence region of the Neuhauser's competing model corresponds to the neutral case, for
 the two-scale multitype contact process, the coexistence region contains cases in which type 1 and type 2 have opposite
 strategies, namely one type is a good invader exploiting the mesoscopic lattice, while the other type is a good
 competitor using the microscopic lattice as its primary resource, and both types live at different time scales.
%  Theorem \ref{coexistence} has a number of important implications in spatial ecology if one thinks of the process as
%  modeling the evolution of two competing species in metapopulations, but this aspect will not be discussed in this
%  article which is devoted to mathematical proofs.

\indent Finally, even through for simplicity we will prove Theorems \ref{survival} and \ref{coexistence} only for the
 two-scale graph depicted in Figure \ref{Fig: graph}, we would like to point out that our proofs easily extend to more
 general graphs, which also gives rise to realistic stochastic spatial models of disease dynamics.
%  Finally, we would like to point out that our proofs of Theorems \ref{survival} and \ref{coexistence} easily
%  extend to more general graphs, even through for simplicity we will prove the results only for the two-scale graph
%  depicted in Figure \ref{Fig: graph}.
%  This gives rise to realistic spatial models of disease dynamics.
 We first describe the general mathematical framework in which our results can be extended and then discuss about the
 relevance of this framework from a biological point of view.
 To begin with, let
 $$ H_1 = (V_1, E_1) \ \ \hbox{and} \ \ H_2 = (V_2, E_2) \quad \hbox{with} \ \
    V_1 \supset V_2 \ \ \hbox{and} \ \ E_1 \cap E_2 = \varnothing $$
 be two infinite graphs.
 We call $H_1$ the microscopic graph and $H_2$ the mesoscopic graph, and consider the two-scale contact process
 evolving on the graph $G = (V_1, E_1 \cup E_2)$ where species $i$ gives birth through the edges of the microscopic graph
 at rate $\beta_i$ and through the edges of the mesoscopic graph at rate $B_i$.
 Assume that we have the following property that we call \emph{separation of the space scales}.
 The mesoscopic graph $H_2$ contains a self-avoiding path $\{X_j : j \in \Z \} \subset V_2$ such that
\begin{enumerate}
 \item For all $j \in \Z$, $H_1$ contains a self-avoiding path of length at least $N$ containing $X_j$. \vspace{4pt}
 \item For all $x, y \in V_2$, the shortest path in $H_1$ connecting $x$ and $y$ has length at least $N$.
\end{enumerate}
 If there is no path in $H_1$ connecting $x$ and $y$ (note that this is the case for the two-scale graph depicted
 in Figure \ref{Fig: graph}) we assume by convention that both vertices are connected by a path of infinite length
 so condition 2 above holds.
 Then, our proof of Theorem \ref{coexistence} implies that, for the two-scale multitype contact process evolving on $G$,
 the Lebesgue measure of the coexistence region is strictly positive provided $N$ is sufficiently large.
 Note that, when $H_1$ is a connected graph, condition 1 above is always satisfied.
 In particular, a natural way to construct a suitable graph $G$ is to start from an infinite connected graph $H_1$, then
 select an infinite subset $V_2 \subset V_1$ of vertices that are at least distance $N$ from each other, and finally add
 enough edges between the vertices in $V_2$ to obtain a mesoscopic graph $H_2$ with at least one infinite self-avoiding path.

\indent Returning to a microscopic graph made of infinitely many finite connected components and thinking of each component
 as a patch, one can legitimately argue that if two patches are connected by an edge of the mesoscopic structure
 then all the vertices of one patch should be connected to all the vertices of the other patch by a mesoscopic edge.
 In a number of contexts, however, patches are arbitrarily large yet adjacent patches are connected through only few
 vertices, in which case our general framework can capture the main features of the dynamics.
 This is the case, for instance, in metastatic diseases, such as malignant tumor cells that first spread within a given
 organ for a long time and then infect quickly other organs while reaching the bloodstream.
 In this context, the connected components of the microscopic structure represent organs or parts of the organs of the
 human body and the mesoscopic structure the vascular system.
 In a different context, one can think of the microscopic structure as a set of major cities and the mesoscopic
 structure as an airline network, where two types of diseases spread: one highly infectious disease such as H1N1
 influenza that spreads quickly through the microscopic structure but slowly through the mesoscopic one due to airport
 screenings (best competitor), and one moderately infectious disease that spreads at an equal speed through the whole
 structure of the network (best invader).

%%%%%%%%%%%%%%%%%%%%%%%%%%%%%%%%%%%%%%%%%%%%%%%%%%%%%%%%%%%%%%%%%%%%%%%%%%%%%%%%%%%%%%%%%%%%%%%%%%%%%%%%%%%%%%%%%%%%%%%%%%%%%%%%%%%%%%%%%%

\section{Proof of Theorem \ref{duality}}
\label{Sec: duality}

\indent Theorem \ref{duality} has been proved by Neuhauser \cite{neuhauser_1992} for the multitype contact process on the
 regular lattice, which corresponds to the case $N = 1$ for our process.
 Her proof relies on duality.
 Thinking of the process as being generated by a graphical representation, the dual process of the multitype contact process
 starting at a space-time point $(x, T)$ exhibits a tree structure that induces an ancestor hierarchy in which the ancestors
 are arranged according to the order they determine the type of the particle at $(x, T)$.
 Her result follows from the existence of a sequence of renewal points dividing the path of the first ancestor into
 independent and identically distributed pieces, as stated in details in Proposition \ref{neuhauser} below.

\indent Similarly, the two-scale contact process is self-dual and the dual process exhibits a tree structure that allows us
 to identify a first ancestor.
 Renewal points can be defined from the topology of the dual process by using the same algorithm as for the multitype contact
 process.
 However, since the graph on which the particles compete is not homogeneous, the space-time displacements between consecutive
 renewal points are no longer identically distributed.
 The key to our proof is to rely on the fact that the graph of Figure \ref{Fig: graph} is invariant by translation of
 vector $u \in N \Z^d$ to show the existence of a subsequence of renewal points performing a random walk.

\indent To define the dual process, we first use an idea of Harris \cite{harris_1972} to construct the two-scale contact
 process graphically from collections of independent Poisson processes.
 These processes are defined for each directed edge $(x, y)$ or vertex $x$ as indicated in Table \ref{Tab-1}.
 The last two columns show the rate of the Poisson processes and the symbols used to construct the graphical representation,
 respectively.
 Unlabeled arrows from $x$ to $y$ indicate birth events: provided site $x$ is occupied and site $y$ empty, $y$ becomes occupied by
 a particle of the same type as the one at $x$.
 The same holds for type 2 arrows if the particle at site $x$ is of type 2, but these arrows are forbidden for type 1 particles,
 which takes into account the selective advantage of type 2.
 Finally, a $\times$ at site $x$ indicates that a particle of either type at this site is killed.
 This graphical representation allows to construct the two-scale multitype contact process starting from any initial configuration.

\begin{table}[t]
\begin{center}
\begin{tabular}{|c|c|c|p{160pt}|}
\hline
 rate & symbols & defined for each \ldots & effect on the configuration of the process \\ \hline
 $B_1$ & $x \,\longrightarrow \,y$ &
       edge $(x, y)$ with $x \leftrightarrow y$ &
       if $y$ is empty, it becomes of the same type as vertex $x$. \\ \hline
 $B_2 - B_1$ & $x \,\overset{2}{\longrightarrow} \,y$ &
       edge $(x, y)$ with $x \leftrightarrow y$ &
       if $x$ is of type 2 and $y$ is empty, $y$ becomes occupied by a type 2 particle. \\ \hline
 $\beta_1$ & $x \,\longrightarrow \,y$ &
       edge $(x, y)$ with $x \sim y$ &
       if $y$ is empty, it becomes of the same type as vertex $x$. \\ \hline
 $\beta_2 - \beta_1$ & $x \,\overset{2}{\longrightarrow} \,y$ &
       edge $(x, y)$ with $x \sim y$ &
       if $x$ is of type 2 and $y$ is empty, $y$ becomes occupied by a type 2 particle. \\ \hline
  1 & $\times$ at vertex $x$ &
       vertex $x \in \Z^d$ &
       if it exists, the particle at vertex $x$ is killed regardless of its type. \\ \hline
\end{tabular}
\end{center}
\caption{\upshape{Harris' graphical representation}}
\label{Tab-1}
\end{table}

\indent The construction of the dual process is done from the graphical representation by ignoring the labels on the arrows.
 We say that there is a path from $(y, T - s)$ to $(x, T)$, which corresponds to a dual path from $(x, T)$ to $(y, T - s)$,
 if there are sequences of times and sites
 $$ s_0 \ = \ T - s \ < \ s_1 \ < \ \cdots \ < \ s_{n + 1} \ = \ T \qquad \hbox{and} \qquad
    x_0 \ = \ y, \,x_1, \,\ldots, \,x_n \ = \ x $$
 such that the following two conditions hold:
\begin{enumerate}
 \item For $i = 1, 2, \ldots, n$, there is an arrow from $x_{i - 1}$ to $x_i$ at time $s_i$ and \vspace{4pt}
 \item For $i = 0, 1, \ldots, n$, the vertical segments $\{x_i \} \times (s_i, s_{i + 1})$ do not contain any $\times$'s.
\end{enumerate}
 The dual process starting at $(x, T)$ is the set-valued process defined by
 $$ \hat \eta_s (x, T) \ = \ \{y \in \Z^d : \hbox{there is a dual path from $(x, T)$ to $(y, T - s)$} \}. $$
 The dual process is naturally defined only for $0 \leq s \leq T$.
 Nevertheless, it is convenient to assume that the Poisson processes in Table \ref{Tab-1} are defined
 for negative times so that the dual process is defined for all $s \geq 0$.
 The reason for introducing the dual process is that it allows us to deduce the presence of a particle at site $x$ at
 time $T$ from the configuration at earlier times by keeping track of potential ancestors.
 This appears in the duality relationship
 $$ \eta_T (x) \ \neq \ 0 \qquad \Longleftrightarrow \qquad
    \eta_{T - s} (y) \ \neq \ 0 \ \ \hbox{for some} \ \ y \in \hat \eta_s (x, T). $$
 The \emph{type} of the particle at vertex $x$ at time $T$ can also be determined from the dual process and the configuration at
 earlier times: the dual process exhibits a tree structure that induces a so-called ancestor hierarchy.
 The members of the dual process at a fixed time (ancestors) are arranged according to the order they determine the type
 of $(x, T)$.

\begin{figure}[t]
\centering
\scalebox{0.40}{\input{lexico.pstex_t}}
\caption{\upshape{Labeled tree structure of the dual process ($d = 1$ and $N = 5$).
 The bold lines refer to the path of the first ancestor which is determined by keeping track of the branch with the
 largest label.
 Numbers at the bottom of the structure indicate the ancestor hierarchy at a given time.
 White squares represent the sequence of renewal points assuming that the tree starting at the last white square at the
 bottom of the picture is infinite.}}
\label{Fig: lexico}
\end{figure}
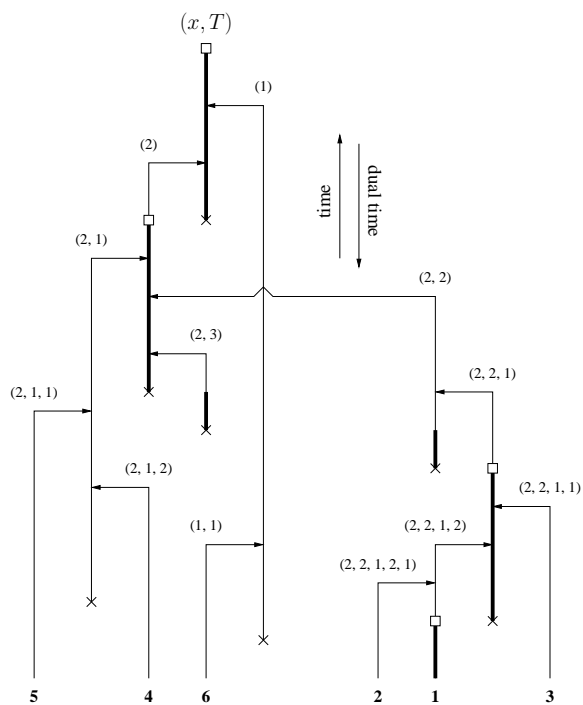

\indent The ancestor hierarchy has been described algorithmically by Neuhauser.
 See the beginning of Section 2 in \cite{neuhauser_1992}.
 The original version of her algorithm is quite intuitive and, although it has been applied in a number of articles, did not
 evolve since then.
 We take advantage of this article to define rigorously the ancestor hierarchy from the tree structure of the dual process,
 following and improving an idea introduced in \cite{lanchier_neuhauser_2006}.
 Even if the dual processes of the contact process and two-scale contact process are different, their topologies are
 similar enough so that our approach applies to both models.
 The idea is to define a function $\phi_s$ that maps the dual process at dual time $s$ into the set of sequences with values
 in $\{1, 2, \ldots \,\} \cup \{\infty \}$ equipped with the lexicographic order $\ll$.
 This function is strictly monotone in the sense that for $x_1, x_2 \in \hat \eta_s (x, T)$,
\begin{equation}
\label{eq: hierarchy}
 \phi_s (x_1) \ \ll \ \phi_s (x_2) \quad \Longleftrightarrow \quad \hbox{$x_2$ comes before $x_1$ in the ancestor hierarchy}.
\end{equation}
 The ancestor hierarchy on the dual process is thus naturally induced by the lexicographic order on the set of sequences through
 the function $\phi_s$.
 Recall that
 $$ (u_1, u_2, \ldots) \ \ll \ (v_1, v_2, \ldots) \quad \hbox{if and only if} \quad
   \left\{\begin{array}{l}
           u_i = v_i \ \ \hbox{for} \ i = 1, 2, \ldots, n - 1 \vspace{3pt} \\ u_n < v_n \end{array} \right. $$
 for some positive integer $n$.
 The function $\phi_s$ corresponds to a labeling of the tree structure of the dual process which is defined inductively by
 going backwards in time.
 The original branch starting at site $x$ at dual time 0 is labeled $(\infty, \infty, \ldots)$.
 Now, assume that a birth event from $x_1$ to $x_2$ occurs at dual time $s_0$ to include site $x_1$ to the dual process, and let
\begin{equation}
\label{eq: parent}
 \phi_{s_0} (x_2) \ = \ (u_1, u_2, \ldots, u_n, \infty, \ldots) \ \ \hbox{with} \ \ u_i < \infty \ \ \hbox{for} \ i = 1, 2, \ldots, n,
\end{equation}
 denote the label on the parent branch (the one at site $x_2$ at dual time $s_0$).
 Then, set
 $$ \phi_{s_0} (x_1) \ = \ (u_1, u_2, \ldots, u_n, m, \infty, \ldots) $$
 if the new branch at site $x_1$ is the $m$th one (going backwards in time) originated from the branch at site $x_2$.
 Also, $\phi_s (x_1) = \phi_{s_0} (x_1)$ until a death mark $\times$ is encountered at site $x_1$ when the site is removed from
 the dual process.
 We refer the reader to Figure \ref{Fig: lexico} for an example of realization of the dual process together with the label
 function along the branches of the tree.
 To avoid cumbersome notations, the sequences are identified to finite dimensional vectors:
 $$ (u_1, u_2, \ldots, u_n, \infty, \ldots) \ \equiv \ (u_1, u_2, \ldots, u_n) $$
 where $u_i < \infty$ for $i = 1, 2, \ldots, n$.
 The label function is ideally suited to keep track of the topology of the dual process and has a number of interesting properties.
 For instance, if \eqref{eq: parent} holds then all the branches in the sub-tree starting at $(x_2, T - s_0)$ have a label
 of the form
 $$ (u_1, u_2, \ldots, u_n, u_{n + 1}, \ldots, u_{n + k}, \infty, \ldots) \ \ \hbox{for some} \ \ u_{n + 1}, \ldots, u_{n + k} < \infty. $$
 In particular, given the labels of two branches, it is straightforward to deduce the label of their most recent common
 ancestor (the root of the minimal sub-tree they both belong to) by looking at the first coordinates they have in common.
 The form of the sequence in \eqref{eq: parent} also indicates that the length of the dual path (number of arrows) from
 $(x, T)$ to $(x_2, T - s_0)$ is equal to $n$, which can be seen as a number of generations.
 Also, it is easy to check that the description of the ancestor hierarchy given in Section 2 of \cite{neuhauser_1992} reduces
 to \eqref{eq: hierarchy} above.

\indent The path of the first ancestor, which is the object of primary interest to determine the type of the particle at $(x, T)$,
 is constructed by following backwards in time the branch with the largest label as defined in condition \eqref{eq: hierarchy}.
 The path of the first ancestor is a complicated object.
 However, in the case of the multitype contact process, Neuhauser \cite{neuhauser_1992} proved that this path can be divided
 into independent and identically distributed pieces at some particular space-time points called renewal points.
 To define renewal points, we say that $(X, \tau)$ lives forever if
\begin{equation}
\label{eq: live_1}
 \hat \eta_s (X, \tau) \ \neq \ \varnothing \qquad \hbox{for all} \ \ s \geq 0.
\end{equation}
 Then, whenever the path of the first ancestor jumps to a point $(X, \tau)$ that lives forever, this point is a renewal point.
 Note that renewal points are well defined only if the starting point $(x, T)$ lives forever, which we assume from now on.
 Let $\{(X_n, \tau_n) : n \geq 0 \}$ denote the sequence of renewal points starting from $(X_0, \tau_0) = (x, 0)$.
 See Figure \ref{Fig: lexico} for a picture.
 For the multitype contact process, which is the case $N = 1$ for our process, Neuhauser \cite{neuhauser_1992} proved
 the following proposition.

\begin{propo}[Neuhauser, 1992]
\label{neuhauser}
 Assume that $N = 1$.
 Then $\{(X_n, \tau_n) : n \geq 0 \}$ performs a random walk on $\Z^d \times \R_+$.
 Moreover, there exist $C_1 < \infty$ and $\gamma_1 > 0$ such that
 $$ P \,(\norm{X_{n + 1} - X_n} > s) \ \leq \ C_1 \,\exp (- \gamma_1 s) \quad \hbox{and} \quad
    P \,(\tau_{n + 1} - \tau_n > s) \ \leq \ C_1 \,\exp (- \gamma_1 s) $$
 for all $n \geq 0$ and all $s \geq 0$.
\end{propo}

\noindent Proposition \ref{neuhauser} is the key to understand the ancestry of a single site and deduce how different
 sites are correlated.
 Unfortunately, it does not hold when $N \geq 3$ in which case the space-time displacements between consecutive renewal
 points are no longer identically distributed due to the geometry of the graph in Figure \ref{Fig: graph}.
 Observing that the parent of a particle located at the corner of a patch has to be closer to the center than the particle
 itself, we see that the path of the first ancestor undergoes a drift directed to the center of the patch.
 The drift is stronger while approaching the boundary of the patch.
 Also, the path of the first ancestor is more likely to jump from one patch to another while approaching the center of
 a patch.
 The key to prove Theorem \ref{duality} is to observe that the graph is invariant by translation of vector $u \in N \Z^d$.
 This implies that the distributions of the dual processes starting at $x$ and $x + u$, respectively, can be deduced from
 one another by a translation of vector $\pm u$, and suggests the existence of subsequences of renewal points satisfying
 Proposition \ref{neuhauser}.
 More precisely, we extract a subsequence inductively by letting
 $$ n_0 \ = \ 0 \qquad \hbox{and} \qquad
    n_{i + 1} \ = \ \min \,\{n > n_i : X_n \in N \Z^d \} \quad \hbox{for} \ i = 0, 1, 2, \ldots $$
 Note that, regardless of $X_n$, the probability that $X_{n + 1} \in N \Z^d$ is bounded from below by a positive constant
 that only depends on $N$.
 This, together with the Borel-Cantelli Lemma and the fact that what happens before and after a certain renewal point
 uses disjoint parts of the graphical representation and so is independent, implies that $n_i$ is almost surely
 finite for all $i$.
 Let
 $$ Y_i \ = \ X_{n_i} \quad \hbox{and} \quad \sigma_i \ = \ \tau_{n_i} \quad \hbox{for all} \ i \geq 0. $$
 Then, $\{(Y_i, \sigma_i) : i \geq 1 \}$ represents the subsequence of renewal points visiting the center of patches
 and we have the following proposition.

\begin{propo}
\label{random_walk}
 The subsequence $\{(Y_i, \sigma_i) : i \geq 1 \}$ is a random walk on $N \Z^d \times \R_+$ regardless of the value
 of $N \geq 1$.
 Moreover, there exist $C_2 < \infty$ and $\gamma_2 > 0$ such that
 $$ P \,(\norm{Y_{i + 1} - Y_i} > s) \ \leq \ C_2 \,\exp (- \gamma_2 s) \quad \hbox{and} \quad
    P \,(\sigma_{i + 1} - \sigma_i > s) \ \leq \ C_2 \,\exp (- \gamma_2 s) $$
 for all $i \geq 1$ and all $s \geq 0$.
\end{propo}
\begin{proof}
 Proposition \ref{neuhauser} indicates that the increments $(X_{n + 1}, \tau_{n + 1}) - (X_n, \tau_n)$ are independent for
 different values of $n$.
 This follows from the fact that, after a renewal point, the path of the first ancestor is determined by the structure of the
 sub-tree starting at the renewal point (rather than the ``super structure'' of the entire dual process) and so depends only
 on parts of the graphical representation that are after the renewal point.
 The same holds for $\{(Y_i, \sigma_i) : i \geq 0 \}$ because it is a subsequence of $\{(X_n, \tau_n) : n \geq 0 \}$.
 Now, since the graphical representation is translation invariant in time and the graph of Figure \ref{Fig: graph} is invariant
 by translation of vector $u \in N \Z^d$, we have that for all $s \geq 0$ and any collection $A$ of subsets of $\Z^d$
 $$ P \,(\hat \eta_s (x, T) - x \in A) \ = \ P \,(\hat \eta_s (y, S) - y \in A) \quad \hbox{whenever} \ \ x - y \in N \Z^d. $$
 In particular, using that $Y_{i + 1} - Y_i \in N \Z^d$ for $i \geq 1$, we obtain
 $$ P \,(\hat \eta_s (Y_i, T - \sigma_i) - Y_i \in A) \ = \
    P \,(\hat \eta_s (Y_{i + 1}, T - \sigma_{i + 1}) - Y_{i + 1} \in A) \quad \hbox{for all} \ \ i \geq 1. $$
 Note however that this does not hold for $i = 0$ since $Y_0 = x$ is not \emph{a priori} at the center of a patch.
 Since $(Y_{i + 1}, \sigma_{i + 1})$ is determined from $\{\hat \eta_s (Y_i, T - \sigma_i) : s \geq 0 \}$, we deduce that
 $$ P \,((Y_{i + 1}, \sigma_{i + 1}) - (Y_i, \sigma_i) \in B) \ = \
    P \,((Y_{i + 2}, \sigma_{i + 2}) - (Y_{i + 1}, \sigma_{i + 1}) \in B) $$
 for any measurable set $B \subset \Z^d \times \R_+$ and all $i \geq 1$.
 In conclusion, the space-time displacements are identically distributed which makes $\{(Y_i, \sigma_i) : i \geq 1 \}$ a random walk.
 To prove the exponential bounds on the space-time displacements, we first observe that, the parameter $N$ being fixed, there
 exists a positive constant $p_N > 0$ such that
 $$ P \,(X_{n + 1} \in N \Z^d \ | \ X_n = y) \ \geq \ p_N \quad \hbox{for all} \ \ y \in \Z^d \ \ \hbox{and all} \ \ n \geq 0. $$
 In words, the next renewal point visiting the center of a patch can be found after at most a geometric number of steps,
 which can be expressed formally as follows:
 $$ P \,(n_{i + 1} - n_i \geq j) \ \leq \ P \,(K \geq j) \quad \hbox{for all} \ \
    j \geq 1 \quad \hbox{where} \quad K \sim \hbox{Geometric} \,(p_N). $$
 Now, members of $\{\tau_{n + 1} - \tau_n : n \geq 0 \}$ are clearly independent.
 Moreover, the same arguments as in the proof of Proposition \ref{neuhauser} that can be found in \cite{neuhauser_1992} imply the
 existence of a continuous random variable $\tau$ with exponentially bounded tails such that
 $$ P \,(\tau_{n + 1} - \tau_n > s) \ \leq \ P \,(\tau > s) \quad \hbox{for all $s > 0$ and $n \geq 0$}. $$
 Let $M (X, \theta)$ stand for the moment generating function of $X$ at point $\theta$.
 Since $K$ is geometrically distributed and temporal increments between consecutive renewal points are independent and uniformly
 stochastically smaller than $\tau$, there exists $\theta_0 > 1$ such that
 $$ \begin{array}{rcl}
     M (\sigma_{i + 1} - \sigma_i, \theta) & = &
     M (\tau_{n_{i + 1}} - \tau_{n_i}, \theta) \ \leq \ M (\tau_K, \theta) \vspace{2pt} \\ & \leq &
    \displaystyle \sum_{n = 1}^{\infty} \ P (K = n) \ M (\tau_1, \theta) \,M (\tau_2, \theta) \ \cdots \ M (\tau_n, \theta) \vspace{-6pt} \\ && \hspace{25pt} \leq \
    \displaystyle \sum_{n = 1}^{\infty} \ P (K = n) \ [M (\tau, \theta)]^n \ = \ M (K, M (\tau, \theta)) \ < \ \infty \end{array} $$
 for all $\theta \in (1, \theta_0)$.
 By using Markov's inequality, we obtain that for $\theta \in (1, \theta_0)$ fixed,
 $$ P \,(\sigma_{i + 1} - \sigma_i > s) \ \leq \
      \theta^{\, -s} \ M (\sigma_{i + 1} - \sigma_i, \theta) \ = \ C_2 \,\exp (- \gamma_2 s) $$
 where $C_2 = M (\sigma_{i + 1} - \sigma_i, \theta) < \infty$ and $\gamma_2 = \log \theta > 0$.
 The exponential bound on the spatial displacements between consecutive renewal points follows from the exponential bound on the
 temporal displacements since the dual process grows at most linearly.
 This follows from the fact that the two-scale contact process is self-dual and that the lengths of all the invasion paths at a given
 time $t$ are uniformly bounded stochastically by $N$ times the radius of the Richardson's model with parameter $\max (B_2, \beta_2)$
 at time $t$.
\end{proof}

\noindent With Proposition \ref{random_walk} in hands, Theorem \ref{duality} in the general case when $N \geq 1$ follows
 from the techniques developed in Sections 2-5 of \cite{neuhauser_1992}.
 The main idea in the neutral case is that the random walk $\{Y_i : i \geq 1 \}$ is recurrent in $d \leq 2$ which implies that
 the paths of the first ancestors of a finite number of sites eventually coalesce with probability 1.
 This makes these sites identical by descent which translates into a clustering of the system.
 In contrast, transience of the random walks in dimension $d \geq 3$ implies that two distinct sites may not be identical by
 descent in which case the states of both sites are independent whenever all the sites are independent at time 0.

%%%%%%%%%%%%%%%%%%%%%%%%%%%%%%%%%%%%%%%%%%%%%%%%%%%%%%%%%%%%%%%%%%%%%%%%%%%%%%%%%%%%%%%%%%%%%%%%%%%%%%%%%%%%%%%%%%%%%%%%%%%%%%%%%%%%%%%%%%

\section{The multitype contact process in finite volume.}
\label{Sec: multitype}

\indent In preparation for proving Theorems \ref{survival} and \ref{coexistence}, we investigate the restriction of the multitype
 contact process to a single patch, i.e., particles landed outside the patch are killed.
 The process is further modified to include spontaneous births of type 1 particles at the center of the patch.
 The reason for introducing spontaneous births is to later compare this process with the two-scale contact process viewed on a
 single patch in which offspring may originate from other patches.
 More precisely, the state of the process at time $t$ is a spatial configuration $\eta_t^N : A_0 \longrightarrow \{0, 1, 2 \}$ where
 $$ A_0 \ = \ (- N / 2, N / 2)^d \,\cap \,\Z^d $$
 is the patch centered at 0.
 The evolution at $x \in A_0$ is described by
 $$ \begin{array}{rcl}
    c_{0 \,\to \,1} (x, \eta^N) & = & \displaystyle \beta_1 \ \sum_{x \sim y} \ \ind \{\eta^N (y) = 1 \} \ + \ 2d \,B_1 \vspace{4pt} \\
    c_{0 \,\to \,2} (x, \eta^N) & = & \displaystyle \beta_2 \ \sum_{x \sim y} \ \ind \{\eta^N (y) = 2 \} \vspace{4pt} \\
    c_{1 \,\to \,0} (x, \eta^N) & = & c_{2 \,\to \,0} (x, \eta^N) \ = \ 1. \end {array} $$
 where $x \sim y$ indicates that $\norm{x - y} = 1$ and $x, y \in A_0$.
 In order to understand both the multitype contact process in finite volume and the two-scale contact process, the first
 step is to compare the processes viewed on suitable length and time scales with oriented percolation.
 To begin with, we introduce 1-dependent oriented percolation with parameter $1 - \ep$ on
 $$ \mathcal G \ = \ \{(z, n) \in \Z^d \times \Z_+ : z^1 + \ldots + z^d + n \hbox{ is even} \} $$
 where $z^i$ denotes the $i$th coordinate of $z \in \Z^d$.
 Each site $(z, n) \in \mathcal G$ is associated with a Bernoulli random variable $\omega (z, n) \in \{0, 1 \}$ with success
 probability $1 - \ep$, and is said to be closed if there is a failure (0) and open if there is a success (1).
 The 1-dependency means that
 $$ P \,(\omega (z_i, n_i) = 1 \ \hbox{for} \ 1 \leq i \leq m) \ = \ (1 - \ep)^m \quad \hbox{whenever} \
      \norm{(z_i, n_i) - (z_j, n_j)} > 1 \ \hbox{for} \ i \neq j. $$
 A site $(z, n)$ is said to be wet (at level $n$) if there exist $z_0, z_1, \ldots, z_n = z$ such that \vspace{-2pt}
\begin{enumerate}
 \item For $i = 0, 1, \ldots, n - 1$, we have $\norm{z_{i + 1} - z_i} = 1$. \vspace{4pt}
 \item For $i = 0, 1, \ldots, n$, site $(z_i, i)$ is open.
\end{enumerate}
 From now on, the parameter $\ep > 0$ is fixed so that
 $$ P \,(W_n \neq \varnothing \ \hbox{for all} \ n \geq 0 \ | \ W_0 = \{0 \}) \ > \ 0, $$
 where $W_n$ denotes the set of wet sites at level $n$.
 To investigate the interacting particle system restricted to a single patch, we let $N = (K + 2) L$ with $K$ and $L$ odd,
 which induces a partition of patch $A_0$ into $L$-cubes, and also consider oriented percolation on
 $$ \begin{array}{rcl}
    \mathcal G_K & = & \mathcal G \,\cap \,\{(- K / 2, K / 2)^d \times \Z_+ \} \vspace{4pt} \\ & = &
    \{(z, n) \in \Z^d \times \Z_+ : z^1 + \ldots + z^d + n \hbox{ is even and} \ \sup_i |z^i| \leq (K - 1) / 2 \}. \end{array} $$
 For the oriented percolation process on $\mathcal G_K$ we denote by $\omega^K (z, n)$ the Bernoulli random variable associated
 to site $(z, n)$ and by $W_n^K$ the corresponding set of wet sites at level $n$.
 The proof of Theorems \ref{survival} and \ref{coexistence} relies on multiscale arguments.
 More precisely, we will consider the following three mesoscopic spatial scales:
\begin{enumerate}
 \item Upper mesoscopic scale.
  The lattice $\Z^d$ is partitioned into $N$-cubes called patches.
  In the next section, we compare the two-scale contact process viewed at the patch level with the oriented percolation
  process on $\mathcal G$ introduced above. \vspace{4pt}
 \item Intermediate mesoscopic scale.
  Each $N$-cube is partitioned into $L$-cubes:
  $$ B_z \ = \ L z \ + \ B_0 \quad \hbox{where} \quad B_0 \ = \ (- L / 2, L / 2)^d \,\cap \,\Z^d \ \hbox{and} \ z \in \Z^d. $$
  In this section, we compare the multitype contact process on patch $A_0$ viewed at the $L$-cube level with the
  oriented percolation process on $\mathcal G_K$ introduced above. \vspace{4pt}
 \item Lower mesoscopic scale.
  Each $L$-cube $B_z$ is further divided into $L^{0.1}$-cubes:
  $$ D_w \ = \ L^{0.1} w \ + \ D_0 \quad \hbox{where} \quad D_0 \ = \ (-L^{0.1} / 2, L^{0.1} / 2)^d \,\cap \,\Z^d \ \hbox{and} \ w \in \Z^d. $$
\end{enumerate}
 Since $A_0$ is finite, type 2 particles go extinct eventually, and the process $\eta_t^N$ converges weakly to a stationary
 distribution with a positive density of type 1.
 The main objective of this section is to prove that when $\beta_2 > \beta_1$ and starting with a significant number of
 type 2 particles the time to extinction grows exponentially with the size of the patch $A_0$.
 Moreover, locally in space and time, type 2 outcompetes type 1 in the sense that, excluding a neighborhood of site 0 whose size
 does not depend on $N$, most of the patch is void of 1's and has a positive density of 2's up to the extinction time.
 To prove these results, the first step is to couple the process with the oriented percolation process on the lattice $\mathcal G_K$
 through the following
\begin{defin}
\label{good}
 Let $B_* = (- L / 6, L / 6)^d$ and $T = L^2$.
 We call $(z, n) \in \mathcal G_K$ a good site if \vspace{-2pt}
\begin{enumerate}
 \item The set $B_z \setminus B_*$ is void of 1's at time $nT$ and \vspace{4pt}
 \item For all $w$ such that $D_w \subset B_z \setminus B_*$, the set $D_w$ contains at least one 2 at time $nT$.
\end{enumerate}
 The set of good sites at level $n$ is denoted by
 $$ X_n^K \ = \ \{z \in \Z^d : (z, n) \in \mathcal G_K \ \hbox{and} \ (z, n) \ \hbox{is good} \} $$
\end{defin}
 Note that site $(0, n)$ has a special treatment: there is no requirement about the spatial configuration of
 the system inside $B_* \subset B_0$.
 The reason is that, due to the presence of spontaneous births, there is a region around the center of the patch which is
 occupied by 1's most of the time.
 Note also that we have excluded boxes $B_z$ with $\sup_i |z^i| = (K + 1) / 2$ which are located along the frontier
 of patch $A_0$ since, due to boundary effects, the density of type 2 particles in these boxes can
 shrink significantly.
 The key result of this section is the following

\begin{propo}
\label{coupling_1}
 Let $\beta_2 > \beta_1 > \beta_c$.
 Then, for $L = L (\ep)$ sufficiently large, the processes can be constructed on the same probability space in such a way that
 $$ P \,(W_n^K \subset X_n^K \ \hbox{for all} \ n \geq 0 \ | \ W_0^K = X_0^K) \ = \ 1. $$
\end{propo}

\noindent In order to establish Proposition \ref{coupling_1}, the objective is to prove that there exists $L = L (\ep)$
 sufficiently large such that the following holds:
\begin{equation}
\label{eq: invasion}
 \begin{array}{l}
  P \,((z_2, n + 1) \hbox{ is good} \ | \ (z_1, n) \hbox{ is good}) \ \geq \ 1 - \ep \vspace{4pt} \\
    \hspace{50pt} \hbox{for all} \ z_1, z_2 \ \hbox{with} \ \norm{z_1 - z_2} = 1 \ \hbox{and} \ (z_1, n), (z_2, n + 1) \in \mathcal G_K.
 \end{array}
\end{equation}
 With inequalities \eqref{eq: invasion} in hands, the proof of Proposition \ref{coupling_1} follows from standard techniques,
 and we refer the reader to the Appendix of \cite{durrett_1995} for more details.
 Since the central site plays a particular role due to the spontaneous births of type 1 particles as pointed out
 in Definition \ref{good}, there are three cases to be considered, as illustrated in Figure \ref{Fig: perco} (the three
 pictures in this figure will be explained later).
 We only prove \eqref{eq: invasion} in the case when $z_2 = 0$ (left picture) which is slightly more complicated than the two
 other cases.
 Using spatial symmetry and homogeneity in time of the evolution rules of the process, it suffices to prove the following

\begin{lemma}
\label{invasion}
 Let $e_1 = (1, 0, \ldots, 0) \in \Z^d$. Then
 $$ P \,((0, 1) \hbox{ is good} \ | \ (e_1, 0) \hbox{ is good}) \ \geq \ 1 - \ep \quad \hbox{for $L$ sufficiently large.} $$
\end{lemma}

\begin{table}[t]
\begin{center}
\begin{tabular}{|c|c|c|p{140pt}|}
\hline
 rate & symbols & defined for each \ldots & effect on the configuration \\ \hline
 $\beta_1$ & $x \,\longrightarrow \,y$ &
       edge $(x, y) \in A_0^2$ with $x \sim y$ &
       if $y$ is empty, it becomes of the same type as vertex $x$. \\ \hline
 $\beta_2 - \beta_1$ & $x \,\overset{2}{\longrightarrow} \,y$ &
       edge $(x, y) \in A_0^2$ with $x \sim y$ &
       if $x$ is of type 2 and $y$ is empty, $y$ becomes occupied by a type 2 particle. \\ \hline
  1 & $\times$ at vertex $x$ &
       vertex $x \in \Z^d$ &
       if it exists, the particle at vertex $x$ is killed regardless of its type. \\ \hline
  $2d B_1$ & $\bullet$ \,at vertex 0 &
       vertex 0 only &
       if $x$ is empty, it becomes occupied by a particle of type 1. \\ \hline
\end{tabular}
\end{center}
\caption{\upshape{Harris' graphical representation}}
\label{Tab-2}
\end{table}

\noindent The proof of Lemma \ref{invasion} relies on duality techniques as well.
 In order to define the dual process, the first step is to construct the multitype contact process on $A_0$ graphically from
 collections of independent Poisson processes \cite{harris_1972}.
 As indicated in Table \ref{Tab-2}, these processes are defined for each directed edge $(x, y) \in A_0 \times A_0$ or
 vertex $x \in A_0$.
 Unlabeled arrows, type 2 arrows, and $\times$'s have the same interpretation as in Table \ref{Tab-1} above.
 The additional symbol $\bullet$ indicates a spontaneous birth of type 1 particle at site 0.
 This graphical representation allows us to construct the multitype contact process on $A_0$ starting from any initial
 configuration.

\indent We say that there is a path from $(y, T - s)$ to $(x, T)$, or equivalently that there is a dual path from $(x, T)$
 to $(y, T - s)$, if there are sequences of times and vertices
 $$ s_0 \ = \ T - s \ < \ s_1 \ < \ \cdots \ < \ s_{n + 1} \ = \ T \qquad \hbox{and} \qquad
    x_0 \ = \ y, \,x_1, \,\ldots, \,x_n \ = \ x $$
 such that the following two conditions hold:
\begin{enumerate}
 \item For $i = 1, 2, \ldots, n$, there is an arrow from $x_{i - 1}$ to $x_i$ at time $s_i$ and \vspace{4pt}
 \item For $i = 0, 1, \ldots, n$, the vertical segments $\{x_i \} \times (s_i, s_{i + 1})$ do not contain any $\times$'s.
\end{enumerate}
 Note that, in our definition of path and dual path, $\bullet$'s have no effect, though they are important in the construction
 of the process.
 The dual process starting at space-time point $(x, T)$ is then defined as the set-valued process
 $$ \hat \eta_s^N (x, T) \ = \ \{y \in \Z^d : \hbox{there is a dual path from $(x, T)$ to $(y, T - s)$} \}. $$
 As previously, it is convenient to assume that the Poisson processes in the graphical representation are defined for negative
 times so that the dual process is defined for all $s \geq 0$.
 To deduce the type of the particle at $(x, T)$ from the configuration at earlier times, we define a labeling of the tree
 structure of the dual process, thus inducing an ancestor hierarchy, by using the algorithm introduced
 in Section \ref{Sec: duality}.
 The path of the first ancestor is then constructed by following backwards in time the branch with the largest label.
 The type of $(x, T)$ is determined as follows:
\begin{enumerate}
 \item If the first ancestor crosses at least one $\bullet$ on its way up to $(x, T)$ \ldots
\begin{enumerate}
 \item regardless of the initial configuration, $(x, T)$ is of type 1. \vspace{4pt}
\end{enumerate}
 \item If the first ancestor does not cross any $\bullet$'s on its way up to $(x, T)$ \ldots
\begin{enumerate}
 \item and lands at time 0 on an empty site, the first ancestor does not determine $(x, T)$.
 \item and lands at time 0 on a 1 and that, on its way up to $(x, T)$, the first ancestor does not cross any type 2 arrow
  then $(x, T)$ is of type 1.
 \item and lands at time 0 on a 1 and that, on its way up to $(x, T)$, the first ancestor crosses a type 2 arrow then the
  first ancestor does not determine the type of $(x, T)$.
  We then follow the path of the first ancestor on its way up to $(x, T)$ until the first 2-arrow we encounter and discard
  all the ancestors of the point where this arrow is directed to.
 \item and lands at time 0 on a 2 then $(x, T)$ is of type 2.
\end{enumerate}
\end{enumerate}
 If the first ancestor does not determine the type of $(x, T)$ (2a and 2c above), we look at the next ancestor in the
 hierarchy, and so on.
 Point 1a follows from the fact that a spontaneous birth of type 1 particle occurs along the path of the first ancestor.
 Points 2a-2d are the same as for the multitype contact process and we refer the reader to \cite{neuhauser_1992},
 page 472, for more details on how to determine the type of $(x, T)$ for the process with no spontaneous birth.
 Finally, note that, since the state space of the process is finite,
 $$ \lim_{s \to \infty} \ \hat \eta_s^N (X, \tau) \ = \ \varnothing \qquad \hbox{with probability 1} $$
 which contradicts the definition of ``living forever'' introduced in \eqref{eq: live_1} above.
 In this section, we say that a space-time point $(X, \tau)$ lives forever if
\begin{equation}
\label{eq: live_2}
 \hat \eta_s^N (X, \tau) \ \neq \ \varnothing \qquad \hbox{for all} \ \ s \leq \tau,
\end{equation}
 that is, condition \eqref{eq: live_1} is only satisfied for all $s \leq \tau$.
 Then, whenever the path of the first ancestor jumps to a space-time point that lives forever in the sense
 of \eqref{eq: live_2}, this point is called a renewal point.
 As in Section \ref{Sec: duality}, the sequence of renewal points divides the path of the first ancestor into independent
 and identically distributed pieces.
 We are now ready to prove Lemma \ref{invasion}. \\

\noindent{\scshape Proof of Lemma \ref{invasion}.}
 Let $x \in B_0 \setminus B_*$ and assume that $(x, T)$ lives forever.
 The first step is to prove that there exist $C_3 < \infty$ and $\gamma_3 > 0$ such that, for all $L$ sufficiently large,
\begin{equation}
\label{eq: type2}
  P \,(\eta_T^N (x) \neq 2 \ | \ (e_1, 0) \ \hbox{is good and} \ (x, T) \ \hbox{lives forever}) \ \leq \ C_3 \,\exp (- \gamma_3 L^{0.1}).
\end{equation}
 To prove \eqref{eq: type2}, we will construct a dual path $\frak S_s$ forbidden for the 1's starting at $(x, T)$ and
 ending at time 0 on a site occupied by a type 2 particle.
 The idea is to apply a modification of the so-called repositioning algorithm (for the original version,
 see page 28 in \cite{durrett_neuhauser_1997}).
 We say that a renewal point is associated with a 2-arrow if the first arrow a particle crosses starting at this renewal
 point and moving up the graphical representation is a 2-arrow.
 We call selected path $\frak A_s$ with origin $x$ and target $y \in B_0$ the following dual path:
 The process starts at $\frak A_0 = x$ and follows the path of the first ancestor starting at $(x, T)$ until the first
 time $\sigma_1$ it jumps to a renewal point associated with a 2-arrow.
 Then, we either leave $\frak A_s$ where it is at that time or reposition it.
 To determine whether and where to reposition the selected path, we denote the location of the second ancestor in the
 hierarchy at time $\sigma_1$, provided this ancestor exists, by $\frak B_{\sigma_1}$.
 Let $m$ be a large constant that does not depend on $L$, and let $\Delta$ be the straight line going through $x$ and $y$.
 Also, we denote the Euclidean distance by $\dist (\,\cdot \,, \,\cdot \,)$.
\begin{enumerate}
 \item Assume that $\frak B_{\sigma_1}$ exists and $(\frak B_{\sigma_1}, T - \sigma_1)$ lives forever. Then
\begin{enumerate}
 \item if $\dist (\frak A_{\sigma_1}, \Delta) > m$
      and $\dist (\frak B_{\sigma_1}, \Delta) < \dist (\frak A_{\sigma_1}, \Delta)$
      set $\frak A_{\sigma_1 +} = \frak B_{\sigma_1}$.
 \item if $\dist (\frak A_{\sigma_1}, \Delta) \leq m$
      and $\dist (\frak B_{\sigma_1}, y) < \dist (\frak A_{\sigma_1}, y)$
      set $\frak A_{\sigma_1 +} = \frak B_{\sigma_1}$.
 \item otherwise, we set $\frak A_{\sigma_1 +} = \frak A_{\sigma_1}$. \vspace{4pt}
\end{enumerate}
 \item Assume that $\frak B_{\sigma_1}$ does not exist or $(\frak B_{\sigma_1}, T - \sigma_1)$ does not live forever. Then
\begin{enumerate}
 \item we set $\frak A_{\sigma_1 +} = \frak A_{\sigma_1}$.
\end{enumerate}
\end{enumerate}
 In either case, we start a new dual process at $(\frak A_{\sigma_1 +}, T - \sigma_1)$ and follow the path of its first
 ancestor until the first time $\sigma_2$ it jumps to a renewal point associated with a 2-arrow when we apply again the
 repositioning algorithm, and so on.
 Intuitively, this causes the selected path $\frak A_s$ to drift towards the target $y$ while staying close to the
 straight line $\Delta$.
 More precisely, let $x_1$ and $x_2$ belong to the segment $(x, y)$ in the order indicated in the following picture:
\begin{center}
 \vspace{10pt}
 \scalebox{0.40}{\input{drift.pstex_t}}
 \vspace{10pt}
\end{center}
 Assume that $\norm{x_1 - x_2} = \sqrt L$ and $\norm{\frak A_{s_0} - x_1} \leq \sqrt L / 4$ at some time $s_0$ and set
 $$ \begin{array}{rcl}
      S_{x_1} & = & \inf \,\{s \geq s_0 : \dist (\frak A_s, x_1) \geq \sqrt L \} \vspace{4pt} \\
      T_{x_2} & = & \inf \,\{s \geq s_0 : \dist (\frak A_s, x_2) < \sqrt L / 4 \}. \end{array} $$
 Then, it can be proved that, for suitable constants $C_4, C_5 < \infty$ and $\gamma_5 > 0$,
\begin{equation}
\label{eq: drift_1}
 P \,(S_{x_1} < T_{x_2} \ \hbox{or} \ T_{x_2} \geq C_4 \sqrt L) \ \leq \ C_5 \,\exp (- \gamma_5 L^{0.2}).
\end{equation}
 The proof of \eqref{eq: drift_1} can be found in \cite{lanchier_neuhauser_2006}, Lemma 3.5.
 That is, starting from the small ball centered at $x_1$ in the picture above, the selected path hits
 the small ball centered at $x_2$ before leaving the large ball centered at $x_1$, this takes less than $C_4 \sqrt L$
 units of time with probability close to 1 when the parameter $L$ is large.
 In particular, if we let
 $$ \begin{array}{rcl}
      S_{\Delta} & = & \inf \,\{s \geq 0 : \dist (\frak A_s, \Delta) \geq \sqrt L \ \hbox{or} \ \dist (\frak A_s, 0) > 2 L \} \vspace{4pt} \\
      T_y & = & \inf \,\{s \geq 0 : \dist (\frak A_s, y) < \sqrt L / 4 \}
    \end{array} $$
 apply \eqref{eq: drift_1} consecutively and use that $\dist (x, y) \leq L \sqrt d$, we obtain
\begin{equation}
\label{eq: drift_2}
  P \,(S_{\Delta} < T_y \ \hbox{or} \ T_y \geq C_4 L \sqrt d) \ \leq \ C_5 \sqrt d \sqrt L \,\exp (- \gamma_5 L^{0.2}).
\end{equation}
 To construct $\frak S_s$, we let $y$ be the corner of $B^* = (- L / 3, L / 3)^d$ closest to $x$ and
 $z = L e_1$ be the center of patch $B_{e_1}$ (see Figure \ref{Fig: perco}).
 The dual path $\frak S_s$ starts at $\frak S_0 = x$, follows the selected path with target $y$ until $T_y$ when
 it hits the Euclidean ball with center $y$ and radius $\sqrt L / 4$, then follows the selected path with target
 $z$ until time $\tau = T - \sqrt L$.
 By \eqref{eq: drift_2},
\begin{equation}
\label{eq: drift_3}
  P \,(\dist (\frak S_s, 0) < \sqrt L \ \hbox{for some} \ s \leq \tau \ \
    \hbox{or} \ \dist (\frak S_{\tau}, z) \geq \sqrt L) \ \leq \ C_6 \,\exp (- \gamma_6 L^{0.2})
\end{equation}
 for suitable $C_6 < \infty$ and $\gamma_6 > 0$.
 In other respects, since $(e_1, 0)$ is good, a straightforward application of Lemma 3.9 in \cite{lanchier_neuhauser_2006}
 implies that
\begin{equation}
\label{eq: drift_4}
\begin{array}{l}
  P \,((\frak S_{\tau}, \sqrt L) \ \hbox{is not occupied by a 2} \ | \dist (\frak S_{\tau}, z) < \sqrt L \vspace{4pt} \\
    \hspace{100pt} \hbox{and} \ (e_1, 0) \ \hbox{is good}) \ \leq \ C_7 \,\exp (- \gamma_7 L^{0.1}) \end{array}
\end{equation}
 for suitable $C_7 < \infty$ and $\gamma_7 > 0$.
 Finally, using the duality properties described above and the fact that 2-arrows are forbidden for the 1's, and
 combining \eqref{eq: drift_3} and \eqref{eq: drift_4}, we obtain
 $$ \begin{array}{l}
  P \,(\eta_T^N (x) \neq 2 \ | \ (e_1, 0) \ \hbox{is good and} \ (x, T) \ \hbox{lives forever}) \vspace{4pt} \\ \hspace{50pt} \leq \
  P \,((\frak S_{\tau}, \sqrt L) \ \hbox{is not occupied by a 2} \ | \ (e_1, 0) \ \hbox{is good}) \vspace{4pt} \\ \hspace{100pt} + \
  P \,(\dist (\frak S_s, 0) < \sqrt L \ \hbox{for some} \ s \leq \tau) \ \leq \ C_3 \,\exp (- \gamma_3 L^{0.1}) \end{array} $$
 for suitable $C_3 < \infty$ and $\gamma_3 > 0$ and all $L$ sufficiently large, which establishes \eqref{eq: type2}.
 The second step is to prove that
\begin{equation}
\label{eq: type1}
  P \,(\eta_T^N (x) = 1 \ | \ (x, T) \ \hbox{does not live forever}) \ \leq \ C_8 \,\exp (- \gamma_8 L)
\end{equation}
 for appropriate $C_8 < \infty$ and $\gamma_8 > 0$.
 Note that, on the event that $(x, T)$ does not live forever, duality implies that the probability in \eqref{eq: type1}
 is equal to 0 for the multitype contact process with no spontaneous birth of type 1 particles.
 In our case, due to the presence of spontaneous births at the central vertex 0, we need to bound the probability that the dual
 process starting at $(x, T)$ hits 0.
 More precisely, since $\norm{x - 0} \geq L / 6$, \eqref{eq: type1} follows from
 $$ \begin{array}{l}
  P \,(0 \in \hat \eta_s^N (x, T) \ \hbox{for some} \ s \leq T \ | \ \hat \eta_T^N (x, T) = \varnothing) \vspace{4pt} \\ \hspace{40pt} \leq \
  P \,(\hbox{the dual process starting at $(x, T)$ has radius} \vspace{4pt} \\
    \hspace{80pt} \hbox{at least $L / 6$} \ | \ (x, T) \ \hbox{does not live forever}) \ \leq \ C_8 \,\exp (- \gamma_8 L)
\end{array} $$
 which is a well-known property of the contact process.
 The proof follows from the analogous result for oriented percolation (see \cite{durrett_1984}, Section 12) and the fact that the
 supercritical contact process viewed on suitable length and time scales dominates oriented
 percolation (see \cite{bezuidenhout_grimmett_1990}).
 Finally, combining \eqref{eq: type2} and \eqref{eq: type1}, we obtain
\begin{equation}
\label{eq: cond1}
 \begin{array}{l}
  P \,(\eta_T^N (x) = 1 \ \hbox{for some} \ x \in B_0 \setminus B_* \ | \ (e_1, 0) \ \hbox{is good}) \vspace{4pt} \\ \hspace{50pt}
    \leq \ C_3 L^d \,\exp (- \gamma_3 L^{0.1}) \ + \ C_8 L^d \,\exp (- \gamma_8 L) \ \leq \ C_9 \,\exp (- \gamma_9 L^{0.1}) \end{array}
\end{equation}
 for suitable $C_9 < \infty$ and $\gamma_9 > 0$, and all $L$ sufficiently large.
 That is, condition 1 in Definition \ref{good} holds with probability arbitrarily close to 1.
 Now, let $D_w \subset B_0 \setminus B_*$.
 Since the process dominates a one-color contact process with parameter $\beta_1 > \beta_c$, we have
\begin{equation}
\label{eq: contact}
 P \,(\eta_T^N (x) = 0 \ \hbox{for all} \ x \in D_w \ | \ (e_1, 0) \ \hbox{is good}) \ \leq \ C_{10} \,\exp (- \gamma_{10} L^{0.1})
\end{equation}
 for appropriate $C_{10} < \infty$ and $\gamma_{10} > 0$.
 Combining \eqref{eq: cond1} and \eqref{eq: contact} implies that
 $$ \begin{array}{l}
  P \,(\eta_T^N (x) \neq 2 \ \hbox{for all} \ x \in D_w \ | \ (e_1, 0) \ \hbox{is good}) \vspace{4pt} \\ \hspace{40pt} \leq \
  P \,(\eta_T^N (x) = 1 \ \hbox{for some} \ x \in D_w \ | \ (e_1, 0) \ \hbox{is good}) \vspace{4pt} \\ \hspace{80pt} + \
  P \,(\eta_T^N (x) = 0 \ \hbox{for all} \ x \in D_w \ | \ (e_1, 0) \ \hbox{is good}) \vspace{4pt} \\ \hspace{40pt} \leq \
  C_9 \,\exp (- \gamma_9 L^{0.1}) \ + \ C_{10} \,\exp (- \gamma_{10} L^{0.1}). \end{array} $$
 In particular, there exist $C_{11} < \infty$ and $\gamma_{11} > 0$ such that
\begin{equation}
\label{eq: cond2}
 \begin{array}{l}
   P \,(\hbox{there exists} \ D_w \subset B_0 \setminus B_* : \eta_T^N (x) \neq 2 \
     \hbox{for all} \ x \in D_w \ | \ (e_1, 0) \ \hbox{is good}) \vspace{4pt} \\ \hspace{40pt} \leq \
   C_9 L^d \,\exp (- \gamma_9 L^{0.1}) \ + \ C_{10} L^d \,\exp (- \gamma_{10} L^{0.1}) \ \leq \
   C_{11} \,\exp (- \gamma_{11} L^{0.1}) \end{array}
\end{equation}
 for $L$ sufficiently large.
 The lemma then follows from \eqref{eq: cond1} and \eqref{eq: cond2}.
\hspace{2mm} $\square$ \\

\begin{figure}[t]
\centering
\scalebox{0.45}{\input{perco.pstex_t}}
\caption{\upshape Picture of the selected path.}
\label{Fig: perco}
\end{figure}
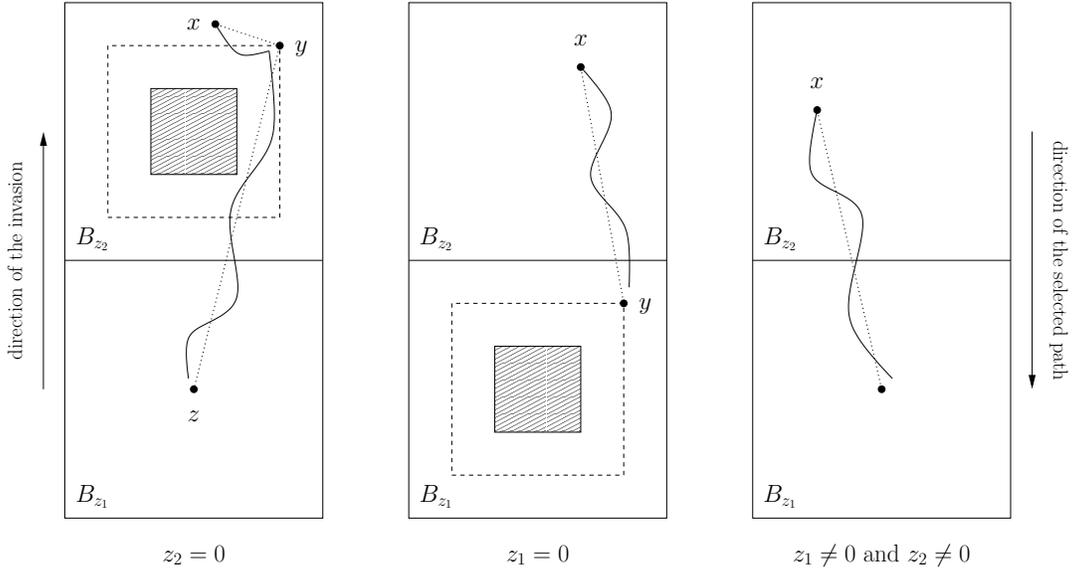

\noindent From Lemma \ref{invasion}, it is easy to deduce \eqref{eq: invasion} when $z_2 = 0$.
 The left picture of Figure \ref{Fig: perco} gives a schematic illustration of the dual path $\frak S_s$ in continuous line.
 The proof of \eqref{eq: invasion} when $z_1 = 0$ and when both $z_1$ and $z_2$ are $\neq 0$ is similar and we refer to the
 last two pictures in Figure \ref{Fig: perco} for an illustration of a suitable dual path in these two cases.
 Note that the case $z_2 = 0$ is slightly more complicated since the repositioning algorithm has to be applied in two different
 directions in order to avoid 0 with high probability.
 With Proposition \ref{coupling_1} in hands, we are now ready to deduce useful properties of the multitype contact process
 restricted to $A_0$ from analogous properties of the oriented percolation process on $\mathcal G_K$.
 We let
 $$ \tau_N \ = \ \inf \,\{t \geq 0 : \eta_t^N (x) \neq 2 \ \hbox{for all} \ x \in A_0 \}. $$
 denote the extinction time of type 2 particles.
 We denote by $P_0$ the conditional probability given the event that vertex 0 (and only vertex 0) is occupied by a
 type 2 particle at time 0.
 In the next lemma, we prove that, starting with a single 2 at the center of the patch, with high probability, the process may
 exhibit only two extreme behaviors: either the 2's spread out successfully and the time to extinction is arbitrarily large
 when $N$ is large, or they die out quickly.

\begin{lemma}
\label{time}
 Let $\beta_2 > \beta_1 > \beta_c$ and $\mathcal I_K = \exp (c K)$.
 Then, there exists $c > 0$ such that
 $$ P_0 \,(t < \tau_N < 3 \mathcal I_K) \ \leq \ C_{12} \,\exp (- \gamma_{12} K) \ + \ C_{13} \,\exp (- \gamma_{13} t) $$
 for all $K$ sufficiently large and suitable constants $C_{12}, C_{13} < \infty$ and $\gamma_{12}, \gamma_{13} > 0$.
\end{lemma}
\begin{proof}
 The proof is divided into two steps, both relying on Proposition \ref{coupling_1}.
 The idea is to decompose the event of interest according to whether the event
 $$ \Omega_K \ = \ \{(X_n^K, n) \not \subset \mathcal G_{K - 2} \ \hbox{for some} \ n \geq 0 \} $$
 occurs or not.
 The occurrence of $\Omega_K$ can be thought of as a successful invasion of type 2.
 We will prove that (i) type 2 particles live an exponentially long time on the event $\Omega_K$, while (ii) they die out quickly
 on the complement of $\Omega_K$. \\
\indent (i) The event $\Omega_K$ occurs.
 In this case, we will prove that
\begin{equation}
\label{eq: i}
 P_0 \,(\Omega_K \ \hbox{and} \ \tau_N < 3 \mathcal I_K) \ \leq \ C_{12} \,\exp (- \gamma_{12} K)
\end{equation}
 for suitable constants $C_{12} < \infty$ and $\gamma_{12} > 0$. Let
 $$ n_0 \ = \ \min \,\{n \geq 0 : (X_n^K, n) \not \subset \mathcal G_{K - 2} \}. $$
 Since the event $\Omega_K$ occurs and the set $X_n^K$ dominates the set of wet sites of a supercritical percolation process,
 there is an in-all-direction expanding region centered at 0 which contains a positive density of good sites.
 Since the correlation between two sites decays exponentially with the distance, the configurations in two $L \times L$ squares
 with no corner in common are almost independent when the parameter $L$ is large, which, together with large deviation estimates
 for the Binomial distribution, implies the existence of a constant $\gamma_{14} > 0$ such that, for $K$ large,
\begin{equation}
\label{eq: slow_1}
 P_0 \,(\Omega_K \ \hbox{and} \ \card X_{n_0}^K < \gamma_{14} K) \ \leq \ C_{15} \,\exp (- \gamma_{15} K)
\end{equation}
 for suitable $C_{15} < \infty$ and $\gamma_{15} > 0$.
 Now, assume that $\card W_{n_0}^K \geq \gamma_{14} K$ and let
 $$ \pi_K \ = \ \inf \,\{n \geq 0 : W_n^K = \varnothing \} $$
 denote the extinction level of the percolation process restricted to $\mathcal G_K$.
 Theorem 2 in \cite{durrett_schonmann_1988} implies that there exists $c > 0$, fixed from now on, such that, for $K$
 sufficiently large,
\begin{equation}
\label{eq: slow_2}
 P \,(\pi_K < 3 T^{-1} \exp (c K) \ \hbox{and} \ \card W_{n_0}^K \geq \gamma_{14} K) \ \leq \ C_{16} \,\exp (- \gamma_{16} K)
\end{equation}
 for suitable $C_{16} < \infty$ and $\gamma_{16} > 0$.
 Combining estimates \eqref{eq: slow_1} and \eqref{eq: slow_2} with the coupling provided in Proposition \ref{coupling_1}
 implies that
 $$ \begin{array}{l}
  P_0 \,(\Omega_K \ \hbox{and} \ \tau_N < 3 \mathcal I_K) \ \leq \
  P_0 \,(\Omega_K \ \hbox{and} \ \card X_{n_0}^K < \gamma_{14} K) \vspace{4pt} \\ \hspace{20pt} + \
  P_0 \,(\tau_N < 3 \mathcal I_K \ \hbox{and} \ \card X_{n_0}^K \geq \gamma_{14} K) \ \leq \
  P_0 \,(\Omega_K \ \hbox{and} \ \card X_{n_0}^K < \gamma_{14} K) \vspace{4pt} \\ \hspace{20pt} + \
  P \,(\pi_K < 3 T^{-1} \mathcal I_K \ \hbox{and} \ \card W_{n_0}^K \geq \gamma_{14} K) \ \leq \
  C_{15} \,\exp (- \gamma_{15} K) \ + \ C_{16} \,\exp (- \gamma_{16} K). \end{array} $$
 This completes the proof of \eqref{eq: i}. \\
\indent (ii) The event $\Omega_K$ does not occur.
 In this case, we will prove that
\begin{equation}
\label{eq: ii}
 P_0 \,(\Omega_K^c \ \hbox{and} \ \tau_N > t) \ \leq \ C_{13} \,\exp (- \gamma_{13} t)
\end{equation}
 for suitable constants $C_{13} < \infty$ and $\gamma_{13} > 0$.
 By Proposition \ref{coupling_1} it suffices to prove the analogous result for the oriented percolation process, namely
\begin{equation}
\label{eq: fast}
  P \,((W_n^K, n) \subset \mathcal G_{K - 2} \ \hbox{for all} \ n \geq 0 \ \hbox{and} \ \pi_K > m \ | \ W_0^K = \{0 \}) \
    \leq \ C_{17} \,\exp (- \gamma_{17} m)
\end{equation}
 for some $C_{17} < \infty$ and $\gamma_{17} > 0$ where $\pi_K$ is the extinction level.
 To establish \eqref{eq: fast}, we couple the restricted and unrestricted percolation processes by letting
\begin{equation}
\label{eq: coupling}
  W_0 \ = \ W_0^K \ = \ \{0 \} \quad \hbox{and} \quad \omega (z, n) \ = \ \omega^K (z, n) \quad \hbox{for all} \ (z, n) \in \mathcal G_K.
\end{equation}
 Assume that there exists $x \in W_n \setminus W_n^K$.
 This together with \eqref{eq: coupling} implies that any open path ending at site $(x, n)$ for the unrestricted percolation
 process leaves the set $\mathcal G_K$.
 Moving up along such a path, we denote by $(x_0, n_0)$ the first site outside $\mathcal G_{K - 2}$ we encounter.
 Using \eqref{eq: coupling} again, it is easy to see that there is an open path ending at $(x_0, n_0)$ for the restricted percolation
 process which implies that $(W_n^K, n) \not \subset G_{K - 2}$ for some $n \geq 0$.
 In conclusion,
\begin{equation}
\label{eq: contradiction}
   W_n \ \neq \ W_n^K \quad \hbox{for some} \ n \geq 0 \quad \Longrightarrow \quad
  (W_n^K, n) \,\not \subset \,\mathcal G_{K - 2} \quad \hbox{for some} \ n \geq 0.
\end{equation}
 Using the reverse of \eqref{eq: contradiction}, we can bound the left-hand side of \eqref{eq: fast} by
 $$ \begin{array}{l}
  P \,((W_n, n) \subset \mathcal G_{K - 2} \ \hbox{for all} \ n \geq 0 \ \hbox{and} \ \pi > m \ | \ W_0 = \{0 \}) \vspace{4pt} \\ \hspace{60pt} \leq \
  P \,(W_n = \varnothing \ \hbox{for some} \ n \geq 0 \ \hbox{and} \ \pi > m \ | \ W_0 = \{0 \}) \vspace{4pt} \\ \hspace{120pt} = \
  P \,(m < \pi < \infty \ | \ W_0 = \{0 \}) \ \leq \ C_{17} \,\exp (- \gamma_{17} m) \end{array} $$
 where $\pi = \inf \,\{n \geq 0 : W_n = \varnothing \}$ which is a well-known property of supercritical oriented percolation
 processes (see page 1031 in \cite{durrett_1984}).
 This establishes \eqref{eq: fast} and \eqref{eq: ii}. \\
\indent We conclude by noticing that
 $$ \begin{array}{rcl}
  P_0 \,(t < \tau_N < 3 \mathcal I_K) & \leq &
  P_0 \,(\Omega_K \ \hbox{and} \ \tau_N < 3 \mathcal I_K) \ + \ P_0 \,(\Omega_K^c \ \hbox{and} \ \tau_N > t) \vspace{4pt} \\ & \leq &
  C_{12} \,\exp (- \gamma_{12} K) \ + \ C_{13} \,\exp (- \gamma_{13} t). \end{array} $$
 This completes the proof.
\end{proof}

\noindent Due to spontaneous births, the region near vertex 0 is mostly occupied by 1's.
 The next two lemmas show however that, provided the 2's invade the patch successfully, the amount of time vertex 0 is occupied by a
 type 2 particle grows exponentially with the size of the patch.

\begin{lemma}
\label{center}
 There exists $p_1 > 0$ that only depends on $L$ such that
 $$ P \,(\eta_t^N (0) = 2 \ \hbox{for all} \ t \in (n T + 1, n T + 2) \ | \ (0, n) \ \hbox{is a good site}) \ \geq \ 2 p_1. $$
\end{lemma}
\begin{proof}
 Let $C$ be the (finite) set of the configurations restricted to $B_0$ such that $(0, n)$ is a good site and, for each
 configuration $\eta \in C$, let $E_{\eta}$ be the set of the realizations of the graphical representation restricted
 to the space-time box $B_0 \times (n T, n T + 2)$ such that
 $$ \{\eta_{nT}^N \equiv \eta \ \hbox{on} \ B_0 \} \,\cap \,E_{\eta} \ \subset \
    \{\eta_t^N (0) = 2 \ \hbox{for all} \ t \in (n T + 1, n T + 2) \}. $$
 Since the events $\{\eta_{nT}^N \equiv \eta \ \hbox{on} \ B_0 \}$ and $E_{\eta}$ are independent, we have
 $$ \begin{array}{l}
  P \,(\eta_t^N (0) = 2 \ \hbox{for all} \ t \in (n T + 1, n T + 2) \ | \ \eta_{nT}^N \equiv \eta \ \hbox{on} \ B_0)
    \vspace{4pt} \\ \hspace{100pt} \
    \geq \ P \,(E_{\eta} \ | \ \eta_{nT}^N \equiv \eta \ \hbox{on} \ B_0) \ = \ P \,(E_{\eta}) \ > \ 0. \end{array} $$
 Using that $C$ is finite, we obtain
 $$ P \,(\eta_t^N (0) = 2 \ \hbox{for all} \ t \in (n T + 1, n T + 2) \ | \ (0, n) \ \hbox{is a good site}) \ \geq \
  \inf_{\eta \in C} \,P^N (E_{\eta}) \ > \ 0. $$
 Finally, since the events $E_{\eta}$ are measurable with respect to the graphical representation restricted to the space-time box
 $B_0 \times (n T, n T + 2)$, the bound $\inf_{\eta \in C} P \,(E_{\eta})$ only depends on $L$.
\end{proof}

\begin{lemma}
\label{occupation}
 Let $\beta_2 > \beta_1 > \beta_c$ and $\mathcal I_K = \exp (c K)$ as in Lemma \ref{time}.
 Also, denote by $\leb$ the Lebesgue measure on the real line.
 Then, for any $s \in (0, \mathcal I_K)$ and $K$ large,
 $$ P_0 \,(\leb \,\{t \in (s, s + \mathcal I_K) : \eta_t^N (0) = 2 \} < K^{-1} \mathcal I_K \ | \ \tau_N \geq 3 \mathcal I_K) \ \leq \ C_{18} \,\exp (- \gamma_{18} K) $$
 for suitable constants $C_{18} < \infty$ and $\gamma_{18} > 0$.
\end{lemma}
\begin{proof}
 The condition $\tau_N \geq 3 \mathcal I_K$ together with Proposition \ref{coupling_1} implies the existence of an in-all-direction expanding
 region centered at 0 which contains a positive density of good sites.
 In particular, there exists $p_2 \in (0, 1)$ such that
 $$ \begin{array}{l}
  P_0 \,(\card \,\{n : n T \in (s, s + \mathcal I_K - 2) \ \hbox{and} \ (0, n) \vspace{4pt} \\
 \hspace{60pt} \hbox{is a good site} \} < p_2 T^{-1} \mathcal I_K \ | \ \tau_N \geq 3 \mathcal I_K) \ \leq \ C_{19} \,\exp (- \gamma_{19} K). \end{array} $$
 From the previous estimate and Lemma \ref{center} it follows that, for $K$ sufficiently large,
 $$ \begin{array}{l}
  P_0 \,(\leb \,\{t \in (s, s + \mathcal I_K) : \eta_t^N (0) = 2 \} < K^{-1} \mathcal I_K \ | \ \tau_N \geq 3 \mathcal I_K) \vspace{4pt} \\
   \hspace{20pt} \leq \ P_0 \,(\leb \,\{t \in (s, s + \mathcal I_K) : \eta_t^N (0) = 2 \} < p_1 p_2 T^{-1} \mathcal I_K \ | \ \tau_N \geq 3 \mathcal I_K) \vspace{4pt} \\
   \hspace{40pt} \leq \ P_0 \,(\card \,\{n : n T \in (s, s + \mathcal I_K - 2) \ \hbox{and} \ \eta_t^N (0) = 2 \ \hbox{for all} \ t \in (n T + 1, n T + 2) \} \vspace{4pt} \\
   \hspace{60pt} < p_1 p_2 T^{-1} \mathcal I_K \ | \ \tau_N \geq 3 \mathcal I_K) \ \leq \ C_{18} \,\exp (- \gamma_{18} K) \end{array} $$
 for suitable $C_{18} < \infty$ and $\gamma_{18} > 0$.
 This completes the proof.
\end{proof}

%%%%%%%%%%%%%%%%%%%%%%%%%%%%%%%%%%%%%%%%%%%%%%%%%%%%%%%%%%%%%%%%%%%%%%%%%%%%%%%%%%%%%%%%%%%%%%%%%%%%%%%%%%%%%%%%%%%%%%%%%%%%%%%%%%%%%%%%%%

\section{Proof of Theorem \ref{survival}}
\label{Sec: survival}

\indent In the previous section, we proved that the two-scale multitype contact process restricted to a single patch viewed
 at the $L$-cube level dominates oriented percolation on $\mathcal G_K$.
 In this section, we rely on consequences of this result, namely Lemmas \ref{time} and \ref{occupation}, to prove that
 the process on the entire lattice viewed at the patch (or $N$-cube) level dominates, in a sense to be specified, oriented
 percolation on $\mathcal G$.
 This will prove in particular Theorem \ref{survival}.
 First of all, we consider the interacting particle system whose state at time $t$ is a function
 $\bar \eta_t : \Z^d \longrightarrow \{0, 1, 2 \}$, and whose evolution at vertex $x \in A_z = N z + A_0$ is described by
 the transition rates
 $$ \begin{array}{rcl}
    c_{0 \,\to \,1} (x, \bar \eta) & = & \displaystyle
         \beta_1 \ \sum_{x \sim y} \ \ind \{\bar \eta (y) = 1 \} \ + \ 2d \,B_1 \vspace{4pt} \\
    c_{0 \,\to \,2} (x, \bar \eta) & = & \displaystyle
         \beta_2 \ \sum_{x \sim y} \ \ind \{\bar \eta (y) = 2 \} \ +
         \ B_2 \ \sum_{x \leftrightarrow y} \ \ind \{\bar \eta (y) = 2 \}
         \ \prod_{w \in A_z} \ \ind \{\bar \eta (w) \neq 2 \} \vspace{4pt} \\
    c_{1 \,\to \,0} (x, \bar \eta) & = & c_{2 \,\to \,0} (x, \bar \eta) \ = \ 1. \end {array} $$
 The dynamics are the same as for the process $\eta_t$ except at the center of the patches in which first type 1 particles
 appear spontaneously at rate $2d B_1$ and second births of type 2 particles originated from adjacent patches are only
 allowed if the patch is void of 2's.
 Note that $2d B_1$ is the rate of spontaneous births of type 1 in the process introduced in Section \ref{Sec: multitype}
 but also an upper bound of the rate at which the center of a patch becomes occupied by a 1 originated from an adjacent
 patch in the two-scale multitype contact process.
 In particular, starting from the same initial configuration, the processes $\eta_t$ and $\bar \eta_t$ can be coupled
 in such a way that
 $$ \{x \in \Z^d : \bar \eta_t (x) = 2 \} \ \subset \ \{x \in \Z^d : \eta_t (x) = 2 \} $$
 so it suffices to prove Theorem \ref{survival} for the process $\bar \eta_t$.
 The process $\bar \eta_t$ viewed at the patch level will be coupled with the oriented percolation process
 on $\mathcal G$ via the following definition.

\begin{defin}
\label{2-stable}
 Let $\mathcal I_K$ as in Lemma \ref{time}.
 Then, site $(z, n) \in \mathcal G$ is said to be type 2 stable if
 $$ \leb \,\{t \in n \mathcal I_K + (0, \mathcal I_K) : \bar \eta_t (N z) = 2 \} \ \geq \ K^{-1} \mathcal I_K. $$
 The set of sites which are type 2 stable at level $n$ is denoted by
 $$ X_n \ = \ \{z \in \Z^d : (z, n) \in \mathcal G \ \hbox{and} \ (z, n) \ \hbox{is type 2 stable} \}. $$
\end{defin}

\noindent The following proposition can be seen as the analog of Proposition \ref{coupling_1}.
 While Proposition \ref{coupling_1} is concerned with the two-scale multitype contact process restricted to a single
 patch viewed at the intermediate mesoscopic scale, Proposition \ref{coupling_2} is concerned with the unrestricted
 process viewed at the upper scale.
 Theorem \ref{survival} is a straightforward consequence of Proposition \ref{coupling_2}.

\begin{propo}
\label{coupling_2}
 Let $\beta_2 > \beta_1 > \beta_c$.
 Then, for $K = K (\ep)$ sufficiently large, the processes can be constructed on the same probability space in such a way that
 $$ P \,(W_n \subset X_n \ \hbox{for all} \ n \geq 0 \ | \ W_0 = X_0) \ = \ 1. $$
\end{propo}
\begin{proof}
 Since the evolution rules of the process are invariant by translation of vector $u \in N \Z^d$, it suffices to prove that
 $$ P \,((e_1, 1) \ \hbox{is type 2 stable} \ | \ (0, 0) \ \hbox{is type 2 stable}) \ \geq \ 1 - \ep $$
 for all $K$ sufficiently large.
 We assume that site $(0, 0)$ is type 2 stable and let
 $$ \sigma_{e_1} \ = \ \inf \,\bigg\{t \geq 0 : \prod_{w \in A_{e_1}}
    \ind \{\bar \eta_s (w) \neq 2 \} = 0 \ \hbox{for all} \ s \in t + (0, 3 \mathcal I_K) \bigg\}. $$
 In words, $\sigma_{e_1}$ is the first time a successful invasion occurs, where successful invasion means that a 2 originated
 from an adjacent patch is sent to $A_{e_1}$ and its family survives at least $3 \mathcal I_K$ units of time in the patch $A_{e_1}$.
 The aim is to prove that $P \,(\sigma_{e_1} > \mathcal I_K)$ is small for $K$ large.
 This, together with Lemma \ref{occupation}, will imply that site $(e_1, 1)$ is type 2 stable with probability arbitrarily close
 to 1 for large enough $K$.
 To estimate the random time $\sigma_{e_1}$ we let $s_0 = 0$ and define by induction
 $$ \begin{array}{l}
     r_i \ = \ \displaystyle \inf \,\bigg\{t \geq s_{i - 1} :
               \displaystyle \prod_{w \in A_{e_1}} \ind \{\bar \eta_t (w) \neq 2 \} = 0 \bigg\}
               \vspace{5pt} \\ \hspace{50pt} \hbox{and} \qquad
     s_i \ = \ \displaystyle \inf \,\bigg\{t \geq r_i :
               \displaystyle \prod_{w \in A_{e_1}} \ind \{\bar \eta_t (w) \neq 2 \} = 1 \bigg\}. \end{array} $$
 In words, $r_i$ is the $i$th time a type 2 originated from an adjacent patch is born at the center of patch $A_{e_1}$
 and $s_i$ the $i$th time patch $A_{e_1}$ becomes void of 2's.
 By letting
 $$ M \ = \ \inf \,\{i \geq 1 : s_i - r_i > 3 \mathcal I_K \}, $$
 we obtain $\sigma_{e_1} = r_M$.
 Let $\bar r_i = s_i - r_i$, and let $\bar s_i$ be the amount of time vertex 0 is occupied by a type 2 between time
 $s_{i - 1}$ and time $r_i$.
 Since vertex 0 is occupied by a 2 at least $K^{-1} \mathcal I_K$ units of time until $\mathcal I_K$ (recall that $(0, 0)$ is type 2 stable),
 on the event $\{r_m > \mathcal I_K \}$, we have
 $$ \sum_{i = 1}^m \ \bar r_i + \bar s_i \ \geq \ K^{-1} \mathcal I_K, $$
 which implies that
 $$ \{r_m > \mathcal I_K \} \ \subset \ \bigcup_{i = 1}^m \ \{\max (\bar r_i, \bar s_i) \geq \mathcal I_K / 2m K \}. $$
 Putting things together, we obtain
\begin{equation}
\label{eq_0}
 \begin{array}{rcl}
  P \,(\sigma_{e_1} > \mathcal I_K) & = &
    \displaystyle \sum_{m = 1}^{\infty} \ P \,(r_m > \mathcal I_K \ \hbox{and} \ M = m) \vspace{5pt} \\ & = &
    \displaystyle \sum_{m = 1}^{\infty} \ \sum_{i = 1}^m \
      P \left(\frac{\mathcal I_K}{2m K} < \bar r_i \leq 3 \mathcal I_K \ \hbox{and} \ M = m \right) \vspace{5pt} \\ & & \hspace{40pt} + \
    \displaystyle \sum_{m = 1}^{\infty} \ \sum_{i = 1}^m \
      P \left(\bar s_i > \frac{\mathcal I_K}{2m K} \ \hbox{and} \ M = m \right). \end{array}
\end{equation}
 We estimate the right-hand side of \eqref{eq_0} in three steps (see \eqref{eq_1}-\eqref{eq_3} below).
 First of all, observing that from time $r_i$ to time $s_i$ the process $\bar \eta_t$ restricted to $A_{e_1}$ evolves
 according to the transition rates of the process $\eta_t^N$ and applying the Markov property, we have
 $$ P \,(M = m \ | \ M \geq m) \ = \ P \,(M = 1) \ = \ P_0 \,(\tau_N > 3 \mathcal I_K) $$
 for all $m \geq 2$, which implies that
 $$ P \,(M = m) \ = \ (1 - p_3)^{m - 1} \,p_3 \quad \hbox{with} \quad p_3 \ = \ P_0 \,(\tau_N > 3 \mathcal I_K) \ > \ 0. $$
 In particular, there is a large $m_{\ep}$, fixed from now on, such that
\begin{equation}
\label{eq_1}
 \sum_{m = m_{\ep}}^{\infty} \ \sum_{i = 1}^m \ P \,(M = m) \ \leq \
 \sum_{m = m_{\ep}}^{\infty} \ m \,(1 - p_3)^{m - 1} \ \leq \ \frac{\ep}{5}.
\end{equation}
 In other respects, Lemma \ref{time} implies that
\begin{equation}
\label{eq_2}
 \begin{array}{l}
  \displaystyle \sum_{m = 1}^{m_{\ep}} \ \sum_{i = 1}^m \ P \left(\frac{\mathcal I_K}{2m K} < \bar r_i \leq 3 \mathcal I_K \right) \
  \leq \ \displaystyle m_{\ep}^2 \ P \left(\frac{\mathcal I_K}{2m_{\ep} K} < \bar r_1 \leq 3 \mathcal I_K \right) \vspace{5pt} \\
  \hspace{50pt} \leq \ \displaystyle m_{\ep}^2 \ P_0 \left(\frac{\mathcal I_K}{2m_{\ep} K} < \tau_N \leq 3 \mathcal I_K \right) \vspace{5pt} \\
  \hspace{100pt} \leq \ m_{\ep}^2 \
  \displaystyle \left[C_{12} \,\exp (- \gamma_{12} K) \ + \ C_{13} \,\exp \bigg(- \frac{\gamma_{13} \mathcal I_K}{2m_{\ep} K} \bigg) \right] \
  \leq \ \frac{\ep}{5} \end{array}
\end{equation}
 for $K$ sufficiently large.
 Finally, letting $X \sim \hbox{Exponential} \,(B_2)$ and using that $N e_1$ is in state 0 a fraction of time of less than $K^{-1}$ with
 probability less than $\ep / 5 m_{\ep}^2$ for $K$ large, we obtain
\begin{equation}
\label{eq_3}
 \begin{array}{l}
  \displaystyle \sum_{m = 1}^{m_{\ep}} \ \sum_{i = 1}^m \ P \left(\bar s_i > \frac{\mathcal I_K}{2m K} \right) \
  \leq \ \displaystyle m_{\ep}^2 \ P \left(\bar s_1 > \frac{\mathcal I_K}{2m_{\ep} K} \right) \vspace{5pt} \\
  \hspace{25pt} \leq \ \displaystyle \frac{\ep}{5} \ + \ m_{\ep}^2 \ P \,\left(X > \frac{\mathcal I_K}{2m_{\ep} K^2} \right) \
  \leq \ \displaystyle \frac{\ep}{5} \ + \ m_{\ep}^2 \ \exp \left(- \,\frac{B_2 \mathcal I_K}{2m_{\ep} K^2} \right) \ \leq \ \frac{2 \ep}{5} \end{array}
\end{equation}
 for $K$ large.
 Applying Lemma \ref{occupation} with $s = \mathcal I_K - \sigma_{e_1}$, we also have
\begin{equation}
\label{eq_4}
 \begin{array}{l}
  P \,(\leb \,\{t \in (\mathcal I_K, 2 \mathcal I_K) : \bar \eta_t (N e_1) = 2 \} < K^{-1} \mathcal I_K \ | \ \sigma_{e_1} \leq \mathcal I_K) \vspace{4pt} \\ \hspace{80pt} = \
  P \,((e_1, 1) \ \hbox{is not type 2 stable} \ | \ \sigma_{e_1} \leq \mathcal I_K) \ \leq \ \displaystyle \frac{\ep}{5} \end{array}
\end{equation}
 for $K$ sufficiently large.
 From \eqref{eq_1}-\eqref{eq_4}, we conclude that
 $$ \begin{array}{l}
     P \,((e_1, 1) \ \hbox{is not type 2 stable} \ | \ (0, 0) \ \hbox{is type 2 stable}) \vspace{4pt} \\ \hspace{40pt} \leq \
     P \,(\sigma_{e_1} > \mathcal I_K) \ + \ P \,((e_1, 1) \ \hbox{is not type 2 stable and} \ \sigma_{e_1} \leq \mathcal I_K) \ \leq \ \ep \end{array} $$
 for $K$ sufficiently large.
 This completes the proof.
\end{proof}

%%%%%%%%%%%%%%%%%%%%%%%%%%%%%%%%%%%%%%%%%%%%%%%%%%%%%%%%%%%%%%%%%%%%%%%%%%%%%%%%%%%%%%%%%%%%%%%%%%%%%%%%%%%%%%%%%%%%%%%%%%%%%%%%%%%%%%%%%%

\section{Proof of Theorem \ref{coexistence}}
\label{Sec: coexistence}

\indent In this section, we prove that the parameter region in which type 1 and type 2 coexist for the two-scale multitype contact
 process has a positive Lebesgue measure, which contrasts with Neuhauser's conjecture about the multitype contact process
 on the regular lattice.
 The strategy of our proof is as follows.
 First of all, we establish coexistence for a particular point of the space of the parameters by comparing the process with
 1-dependent oriented percolation.
 Interestingly, survival of type 1 and survival of type 2 are proved by considering different time scales.
 In other words, the process will be simultaneously coupled with two different oriented percolation processes, one following the
 evolution of type 1 particles at a certain time scale, the other one following the evolution of type 2 particles at a slower
 time scale.
 The suitable time scale for type 2 is fixed afterward and depends on the time scale chosen for type 1.
 In both cases, however, the process is viewed at the same spatial scale, namely the upper mesoscopic scale (patch level).
 Since our proof relies on a block construction, standard perturbation arguments imply that the coexistence region can be
 extended to an open set containing the coexistence point, which proves Theorem \ref{coexistence}.
 In order to compare the process with oriented percolation, we introduce the following definition.

\begin{defin}
\label{1-stable}
 Let $\mathcal I_K$ as in Lemma \ref{time} and $\mathcal J > 0$.
 Then, site $(z, n) \in \mathcal G$ is said to be
\begin{enumerate}
 \item type 1 stable whenever vertex $N z$ is occupied by a 1 at time $n \mathcal J$. \vspace{4pt}
 \item type 2 stable whenever $\leb \,\{t \in n \mathcal I_K + (0, \mathcal I_K) : \eta_t (N z) = 2 \} \geq K^{-1} \mathcal I_K$.
\end{enumerate}
 For $i = 1, 2$, the set of type $i$ stable sites at level $n$ is denoted by
 $$ X_n^i \ = \ \{z \in \Z^d : (z, n) \in \mathcal G \ \hbox{and} \ (z, n) \ \hbox{is type $i$ stable} \}. $$
\end{defin}

\noindent Note that the definition of type 2 stable is slightly different from the one in Definition \ref{2-stable} in that
 it now applies to events related to the two-scale multitype contact process instead of the modified process introduced in
 Section \ref{Sec: survival}.
 However, Proposition \ref{coupling_2} still holds for $X_n^2$ since the set of 2's in the two-scale multitype contact process
 dominates the set of 2's in the modified process.
 To exhibit a point of the space of the parameters at which coexistence occurs, we fix
\begin{equation}
\label{eq: parameters}
 B_1 > 0 \qquad B_2 > 0 \qquad \beta_1 = 0 \qquad \beta_2 > \beta_c \qquad \delta_2 = 1.
\end{equation}
 The condition $\delta_2 = 1$ is to fix the time scale.
 The condition $\beta_1 = 0$ indicates that type 1 particles can only survive by jumping from patch to patch.
 In particular, to prove that they survive, the idea is to choose $\delta_1 > 0$ so small that a 1 at the center of a patch
 can produce and send its offspring to adjacent patches before being killed.
 More importantly, since $\beta_1 = 0$, survival of type 1 particles does not depend on the patch size.
 Coexistence is then obtained by choosing $N$ so large that type 2 particles can establish themselves an arbitrarily long
 time in a single patch.
 Since type 1 particles have a positive death rate, centers of patches are empty a positive fraction of time which allows
 type 2 particles to survive by invading adjacent patches from time to time.

\begin{propo}
\label{coupling_3}
 Assume \eqref{eq: parameters}.
 Then, for suitable $\mathcal J < \infty$, $\delta_1 > 0$ and $N < \infty$, the processes can be constructed on the same
 probability space in such a way that
 $$ P \,(W_n^1 \subset X_n^1 \ \hbox{and} \ W_n^2 \subset X_n^2 \ \hbox{for all} \ n \geq 0 \ |
      \ W_0^1 = X_0^1 \ \hbox{and} \ W_0^2 = X_0^2) \ = \ 1 $$
 where $W_n^1$ and $W_n^2$ are two copies of $W_n$.
\end{propo}
\begin{proof}
 As in Proposition \ref{coupling_2}, it suffices to prove that, for $i = 1, 2$,
\begin{equation}
\label{eq: invade}
  P \,((e_1, 1) \ \hbox{is type $i$ stable} \ | \ (0, 0) \ \hbox{is type $i$ stable}) \ \geq \ 1 - \ep
\end{equation}
 for a suitable choice of the parameters.
 The first step is to show that there is enough room for type 1 particles to invade the center of the patch $A_{e_1}$.
 By observing that
 $$ c_{0 \,\to \,2} (N e_1, \eta) \ \leq \ 2d \,(B_2 + \beta_2) \quad \hbox{and} \quad
    c_{2 \,\to \,0} (N e_1, \eta) \ = \ \delta_2 \ = \ 1 $$
 we obtain
 $$ \limsup_{s \to \infty} \ \E \,[s^{-1} \,\leb \,\{t \in (0, s) : \eta_t (N e_1) = 2 \}] \ \leq \
    \Theta_2 \ := \ \frac{2d \,(B_2 + \beta_2)}{1 + 2d \,(B_2 + \beta_2)} \ < \ 1. $$
 Let $\ep_0 \in (0, 1 - \Theta_2)$.
 Large deviation estimates for the Poisson distribution give
 $$ P \,(\leb \,\{t \in (0, \mathcal J) : \eta_t (N e_1) = 2 \} > (\Theta_2 + \ep_0) \mathcal J) \ \leq \ C_{20} \,\exp (- \gamma_{20} \mathcal J) $$
 for suitable $C_{20} < \infty$ and $\gamma_{20} > 0$.
 In particular,
 $$ \begin{array}{l}
  P \,(\eta_t (N e_1) \neq 1 \ \hbox{for all} \ t \in (0, \mathcal J) \ \hbox{and} \ \eta_t (0) = 1 \ \hbox{for all} \ t \in (0, \mathcal J)) \vspace{4pt} \\ \hspace{20pt} \leq \
  P \,(\leb \,\{t \in (0, \mathcal J) : \eta_t (N e_1) = 2 \} > (\Theta_2 + \ep_0) \mathcal J) \vspace{4pt} \\ \hspace{40pt} + \
  P \,(\eta_t (N e_1) \neq 1 \ \hbox{for all} \ t \in (0, \mathcal J) \ \hbox{and} \ \eta_t (0) = 1 \ \hbox{for all} \ t \in (0, \mathcal J) \ \hbox{and} \vspace{4pt} \\ \hspace{60pt}
    \ \leb \,\{t \in (0, \mathcal J) : \eta_t (N e_1) = 2 \} \leq (\Theta_2 + \ep_0) \mathcal J) \vspace{4pt} \\ \hspace{20pt}
    \leq \ C_{20} \,\exp (- \gamma_{20} \mathcal J) \ + \ \exp (- B_1 (1 - \Theta_2 - \ep_0) \mathcal J). \end{array} $$
 Finally, taking $\mathcal J$ large and then $\delta_1 > 0$ small, we get
 $$ \begin{array}{l}
  P \,((e_1, 1) \ \hbox{is not type 1 stable} \ | \ (0, 0) \ \hbox{is type 1 stable}) \ = \
  P \,(\eta_{\mathcal J} (N e_1) \neq 1 \ | \ \eta_0 (0) = 1) \vspace{4pt} \\ \hspace{20pt} \leq \
  P \,(\eta_{\mathcal J} (N e_1) \neq 1 \ \hbox{and} \ \eta_t (N e_1) = 1 \ \hbox{for some} \ t \in (0, \mathcal J) \ | \ \eta_0 (0) = 1) \vspace{4pt} \\ \hspace{40pt} + \
  P \,(\eta_t (N e_1) \neq 1 \ \hbox{for all} \ t \in (0, \mathcal J) \ \hbox{and} \ \eta_t (0) = 1 \ \hbox{for all} \ t \in (0, \mathcal J) \ | \ \eta_0 (0) = 1) \vspace{4pt} \\ \hspace{60pt} + \
  P \,(\eta_t (0) \neq 1 \ \hbox{for some} \ t \in (0, \mathcal J) \ | \ \eta_0 (0) = 1)  \vspace{4pt} \\ \hspace{20pt} \leq \
  C_{20} \,\exp (- \gamma_{20} \mathcal J) \ + \ \exp (- B_1 (1 - \Theta_2 - \ep_0) \mathcal J) \ + \ 2 \,(1 - \exp (- \delta_1 \mathcal J)) \ \leq \ \ep \end{array} $$
 which establishes \eqref{eq: invade} for $i = 1$.
 Moreover, survival of type 1 particles holds regardless of the patch size so the proof that conditions \eqref{eq: invade}
 for $i = 1$ and $i = 2$ hold \emph{simultaneously} for the same parameters follows by taking $L$ and $K$ sufficiently large, and
 applying the results of the previous two sections.
 This proves that both types coexist.
 To conclude, we briefly justify the fact that the results of the previous sections hold as well under the new assumptions
\begin{equation}
\label{eq: assumptions}
  \beta_1 \ = \ 0 \qquad \hbox{and} \qquad \delta_1 \ > \ 0.
\end{equation}
 First, the condition $\beta_1 = 0$ implies that, starting from any initial configuration,
 $$ P \,(\eta_t^N (x) = 1 \ \hbox{for some} \ x \in B_z \setminus \{0 \} \ \hbox{and} \ t \geq \sqrt L) \ \leq \ L^d \,\exp (- \delta_1 \sqrt L)$$
 which, now that $\delta_1 > 0$ is fixed, can be made arbitrarily small by taking $L$ large.
 In words, except at the center of the patch, all the 1's in $B_z \subset A_0$ are rapidly killed, so the proof of Lemma \ref{invasion}
 extends easily under \eqref{eq: assumptions}.
 Proposition \ref{coupling_1} and Lemma \ref{time} follow as well.
 Now, we observe that the condition $\delta_1 > 0$ implies that vertex 0 is empty a positive fraction of time, which allows for
 invasions of particles of type 2 at the center of the patch.
 With this in mind, one can check easily that the proofs of Lemmas \ref{center} and \ref{occupation} also apply under the
 assumptions \eqref{eq: assumptions}.
 Note however that the new lower bound $p_1 > 0$ in Lemma \ref{center} might be smaller.
 Finally, the proof of Proposition \ref{coupling_2} still holds as a consequence of Lemmas \ref{time} and \ref{occupation}.
\end{proof}

%%%%%%%%%%%%%%%%%%%%%%%%%%%%%%%%%%%%%%%%%%%%%%%%%%%%%%%%%%%%%%%%%%%%%%%%%%%%%%%%%%%%%%%%%%%%%%%%%%%%%%%%%%%%%%%%%%%%%%%%%%%%%%%%%%%%%%%%%%

\end{document}

%% file: lexico.pstex_t
\begin{picture}(0,0)%
\includegraphics{lexico.pstex}%
\end{picture}%
\setlength{\unitlength}{3947sp}%
\begingroup\makeatletter\ifx\SetFigFontNFSS\undefined%
\gdef\SetFigFontNFSS#1#2#3#4#5{%
  \reset@font\fontsize{#1}{#2pt}%
  \fontfamily{#3}\fontseries{#4}\fontshape{#5}%
  \selectfont}%
\fi\endgroup%
\begin{picture}(9021,10893)(-98,-10051)
\put(3001,539){\makebox(0,0)[b]{\smash{{\SetFigFontNFSS{20}{24.0}{\familydefault}{\mddefault}{\updefault}$(x, T)$}}}}
\end{picture}%

%% file: drift.pstex_t
\begin{picture}(0,0)%
\includegraphics{drift.pstex}%
\end{picture}%
\setlength{\unitlength}{3947sp}%
\begingroup\makeatletter\ifx\SetFigFontNFSS\undefined%
\gdef\SetFigFontNFSS#1#2#3#4#5{%
  \reset@font\fontsize{#1}{#2pt}%
  \fontfamily{#3}\fontseries{#4}\fontshape{#5}%
  \selectfont}%
\fi\endgroup%
\begin{picture}(9624,4814)(589,-1568)
\put(4201,2039){\makebox(0,0)[b]{\smash{{\SetFigFontNFSS{20}{24.0}{\familydefault}{\mddefault}{\updefault}$x_1$}}}}
\put(6601,2039){\makebox(0,0)[b]{\smash{{\SetFigFontNFSS{20}{24.0}{\familydefault}{\mddefault}{\updefault}$x_2$}}}}
\put(1201,2039){\makebox(0,0)[b]{\smash{{\SetFigFontNFSS{20}{24.0}{\familydefault}{\mddefault}{\updefault}$x$}}}}
\put(9601,2039){\makebox(0,0)[b]{\smash{{\SetFigFontNFSS{20}{24.0}{\familydefault}{\mddefault}{\updefault}$y$}}}}
\put(3226, 89){\rotatebox{45.0}{\makebox(0,0)[b]{\smash{{\SetFigFontNFSS{20}{24.0}{\familydefault}{\mddefault}{\updefault}$\sqrt L$}}}}}
\put(8401,989){\makebox(0,0)[b]{\smash{{\SetFigFontNFSS{20}{24.0}{\familydefault}{\mddefault}{\updefault}$\Delta$}}}}
\end{picture}%

%% file: perco.pstex_t
\begin{picture}(0,0)%
\includegraphics{perco.pstex}%
\end{picture}%
\setlength{\unitlength}{3947sp}%
\begingroup\makeatletter\ifx\SetFigFontNFSS\undefined%
\gdef\SetFigFontNFSS#1#2#3#4#5{%
  \reset@font\fontsize{#1}{#2pt}%
  \fontfamily{#3}\fontseries{#4}\fontshape{#5}%
  \selectfont}%
\fi\endgroup%
\begin{picture}(14829,7921)(-809,-7070)
\put(-599,-2761){\rotatebox{90.0}{\makebox(0,0)[b]{\smash{{\SetFigFontNFSS{17}{20.4}{\familydefault}{\mddefault}{\updefault}direction of the invasion}}}}}
\put(13801,-2761){\rotatebox{270.0}{\makebox(0,0)[b]{\smash{{\SetFigFontNFSS{17}{20.4}{\familydefault}{\mddefault}{\updefault}direction of the selected path}}}}}
\put(1801,464){\makebox(0,0)[b]{\smash{{\SetFigFontNFSS{20}{24.0}{\familydefault}{\mddefault}{\updefault}$x$}}}}
\put(7201,239){\makebox(0,0)[b]{\smash{{\SetFigFontNFSS{20}{24.0}{\familydefault}{\mddefault}{\updefault}$x$}}}}
\put(10501,-361){\makebox(0,0)[b]{\smash{{\SetFigFontNFSS{20}{24.0}{\familydefault}{\mddefault}{\updefault}$x$}}}}
\put(1801,-6961){\makebox(0,0)[b]{\smash{{\SetFigFontNFSS{20}{24.0}{\familydefault}{\mddefault}{\updefault}$z_2 = 0$}}}}
\put(6601,-6961){\makebox(0,0)[b]{\smash{{\SetFigFontNFSS{20}{24.0}{\familydefault}{\mddefault}{\updefault}$z_1 = 0$}}}}
\put(11401,-6961){\makebox(0,0)[b]{\smash{{\SetFigFontNFSS{20}{24.0}{\familydefault}{\mddefault}{\updefault}$z_1 \neq 0$ and $z_2 \neq 0$}}}}
\put(151,-2536){\makebox(0,0)[lb]{\smash{{\SetFigFontNFSS{20}{24.0}{\familydefault}{\mddefault}{\updefault}$B_{z_2}$}}}}
\put(151,-6136){\makebox(0,0)[lb]{\smash{{\SetFigFontNFSS{20}{24.0}{\familydefault}{\mddefault}{\updefault}$B_{z_1}$}}}}
\put(4951,-6136){\makebox(0,0)[lb]{\smash{{\SetFigFontNFSS{20}{24.0}{\familydefault}{\mddefault}{\updefault}$B_{z_1}$}}}}
\put(4951,-2536){\makebox(0,0)[lb]{\smash{{\SetFigFontNFSS{20}{24.0}{\familydefault}{\mddefault}{\updefault}$B_{z_2}$}}}}
\put(9751,-2536){\makebox(0,0)[lb]{\smash{{\SetFigFontNFSS{20}{24.0}{\familydefault}{\mddefault}{\updefault}$B_{z_2}$}}}}
\put(9751,-6136){\makebox(0,0)[lb]{\smash{{\SetFigFontNFSS{20}{24.0}{\familydefault}{\mddefault}{\updefault}$B_{z_1}$}}}}
\put(1801,-5011){\makebox(0,0)[b]{\smash{{\SetFigFontNFSS{20}{24.0}{\familydefault}{\mddefault}{\updefault}$z$}}}}
\put(3301,164){\makebox(0,0)[b]{\smash{{\SetFigFontNFSS{20}{24.0}{\familydefault}{\mddefault}{\updefault}$y$}}}}
\put(8101,-3436){\makebox(0,0)[b]{\smash{{\SetFigFontNFSS{20}{24.0}{\familydefault}{\mddefault}{\updefault}$y$}}}}
\end{picture}%

%% file: M2CP.bbl
\begin{thebibliography}{10}
\small

\bibitem{belhadji_lanchier_2008}
 Belhadji, L. and Lanchier, N. (2008).
 Two-scale contact process and effects of habitat fragmentation on metapopulations.
\emph{Markov Process. Related Fields}, \textbf{14} 487--514.

\bibitem{bezuidenhout_grimmett_1990}
 Bezuidenhout, C. and Grimmett, G. R. (1990).
 The critical contact process dies out.
\emph{Ann. Probab.} \textbf{18} 1462--1482.

\bibitem{chan_durrett_2006}
 Chan, B. and Durrett, R. (2006).
 A new coexistence result for competing contact processes.
\emph{Ann. Appl. Probab.} \textbf{16} 1155--1165.

\bibitem{chan_durrett_lanchier_2009}
 Chan, B., Durrett, R. and Lanchier, N. (2009).
 Coexistence for a multitype contact process with seasons.
\emph{Ann. Appl. Probab.}, \textbf{19} 1921--1943.

\bibitem{durrett_1984}
 Durrett, R. (1984).
 Oriented percolation in two dimensions.
\emph{Ann. Probab.} \textbf{12} 999--1040.

\bibitem{durrett_1995}
 Durrett, R. (1995).
 Ten lectures on particle systems.
 In \emph{Lectures on probability theory (Saint-Flour, 1993)},
 volume 1608 of \emph{Lecture Notes in Math.}, pages 97--201.
 Springer, Berlin.

\bibitem{durrett_lanchier_2008}
 Durrett, R. and Lanchier, N. (2008).
 Coexistence in host-pathogen systems.
\emph{Stochastic Process. Appl.} \textbf{118} 1004--1021.

\bibitem{durrett_neuhauser_1997}
 Durrett, R. and Neuhauser, C. (1997).
 Coexistence results for some competition models.
\emph{Ann. Appl. Probab.} \textbf{7} 10--45.

\bibitem{durrett_schonmann_1988}
 Durrett, R. and Schonmann, R. (1988).
 The contact process on a finite set II.
\emph{Ann. Probab.} \textbf{16} 1570--1583.

\bibitem{hanski_1999}
 Hanski, I. (1999).
\emph{Metapopulation Ecology.}
 Oxford University Press.

\bibitem{harris_1972}
 Harris, T. E. (1972).
 Nearest neighbor Markov interaction processes on multidimensional lattices.
\emph{Adv. Math.} \textbf{9} 66--89.

\bibitem{harris_1974}
 Harris, T. E. (1974).
 Contact interactions on a lattice.
\emph{Ann. Probability} \textbf{2} 969--988.

% \bibitem{huffaker_1958}
%  Huffaker, C. B. (1958).
%  Experimental studies on predation: Dispersion factors and predator-prey oscillations.
% \emph{Hilgardia} \textbf{27} 342--383.

\bibitem{lanchier_neuhauser_2006}
 Lanchier, N. and Neuhauser, C. (2006).
 A spatially explicit model for competition among specialists and generalists in a heterogeneous environment.
\emph{Ann. Appl. Probab.} \textbf{16} 1385--1410.

\bibitem{levins_1969}
 Levins, R. (1969).
 Some demographic and genetic consequences of environmental heterogeneity for biological control.
\emph{Bulletin of the Entomological Society of America} \textbf{15} 237--240.

\bibitem{neuhauser_1992}
 Neuhauser, C. (1992).
 Ergodic theorems for the multitype contact process.
\emph{Prob. Theory Relat. Fields} \textbf{91} 467--506.

\bibitem{pemantle_1992}
 Pemantle, R. (1992).
 The contact process on trees.
\emph{Ann. Probab.} \textbf{20} 2089--2116.

\end{thebibliography}
